\documentclass[a4paper, 12pt]{article}
\pdfoutput=1

\usepackage{amsfonts,amsmath,amssymb,mathtools,dsfont,microtype} 
\usepackage{relsize}
\usepackage[dvipsnames]{xcolor} 
\usepackage[noBBpl]{mathpazo} 

\DeclareFontFamily{U}{cmr}{}

\DeclareFontShape{U}{cmr}{m}{n}{
	<-6> cmr5
	<6-7> cmr6
	<7-8> cmr7
	<8-9> cmr8
	<9-10> cmr9
	<10-12> cmr10
	<12-> cmr12}{}

\DeclareSymbolFont{Xcmr} {U} {cmr}{m}{n}
\DeclareMathSymbol{\Delta}{\mathord}{Xcmr}{'001}
\DeclareMathSymbol{\Upsilon}{\mathord}{Xcmr}{'007}
\DeclareMathSymbol{\Omega}{\mathord}{Xcmr}{'012}

\usepackage[labelsep = period, labelfont = bf, justification = centering]{caption}
\usepackage{float,graphicx,subcaption}
\DeclareCaptionSubType*[roman]{figure}

\usepackage{booktabs,makecell}
\usepackage{enumitem,mdwlist}
\setlist[itemize]{topsep=0ex,itemsep=0ex,parsep=0.4ex}
\setlist[enumerate]{topsep=0ex,itemsep=0ex,parsep=0.4ex}

\usepackage{pgf,tikz,ifthen,calc}
\usepackage{tkz-euclide}
\tkzSetUpPoint[fill=black, size = 3pt]
\tkzSetUpLine[color=black, line width=0.6pt]
\tkzSetUpCircle[color=black, line width=0.6pt]

\usepackage[left = 2.5cm, right = 2.5cm, top = 2.5cm, bottom = 2.5cm, headsep = 12pt, headheight = 15pt]{geometry}
\usepackage{parskip}

\usepackage[hyphens]{url} 
\usepackage[linktoc = all, hidelinks, colorlinks, unicode=true, pagebackref=true, hypertexnames=false]{hyperref} 
\usepackage[capitalise, compress, nameinlink, noabbrev]{cleveref} 

\hypersetup{linkcolor={blue!70!black}, citecolor={green!70!black}, urlcolor={blue!70!black}}

\usepackage{amsthm,thmtools,thm-restate}

\makeatletter
\@ifundefined{newcounteralias}{}{%
  \renewcommand\thmt@autorefsetup{%
    \@xa\def
    \csname\thmt@envname autorefname\@xa\endcsname
    \@xa{\thmt@thmname}%
  }%
}
\makeatother

\declaretheorem[name = Theorem, numberwithin = section, style = plain]{theorem}
\declaretheorem[name = Claim, numberwithin = theorem, style = plain]{claim}
\declaretheorem[name = Corollary, numberlike = theorem, style = plain]{corollary}

\declaretheorem[name = Definition, numberlike = theorem, style = definition]{definition}

\declaretheorem[name = Lemma, numberlike = theorem, style = plain]{lemma}
\declaretheorem[name = Problem, numberlike = theorem, style = plain]{problem}

\declaretheorem[name = Proposition, numberlike = theorem, style = plain]{proposition}

\declaretheorem[name = Remark, style = definition]{remark*}
 
\crefname{theorem}{Theorem}{Theorems}
\crefname{lemma}{Lemma}{Lemmas}
\crefname{conjecture}{Conjecture}{Conjectures}
\crefname{claim}{Claim}{Claims}
\crefname{corollary}{Corollary}{Corollaries}
\crefname{fact}{Fact}{Facts}
\crefname{observation}{Observation}{Observations}
\crefname{problem}{Problem}{Problems}
\crefname{proposition}{Proposition}{Propositions}
\crefname{remark}{Remark}{Remarks}
\crefname{remark*}{Remark}{Remarks}
\crefname{section}{\S}{\S\S} 
\crefname{subsection}{\S}{\S\S} 
\newenvironment{proofclaim}[1][Proof of claim]{\begin{proof}[#1]}{\end{proof}}

\DeclareFontFamily{U}{matha}{\hyphenchar\font45}
\DeclareFontShape{U}{matha}{m}{n}{
	<5> <6> <7> <8> <9> <10> gen * matha
	<10.95> matha10 <12> <14.4> <17.28> <20.74> <24.88> matha12
}{}
\DeclareSymbolFont{matha}{U}{matha}{m}{n}

\DeclareMathSymbol{\specialuparrow}{\mathrel}{matha}{"D2}
\DeclareMathSymbol{\specialrightarrow}{\mathrel}{matha}{"D1}

\renewcommand*{\backref}[1]{}
\renewcommand*{\backrefalt}[4]{
	\ifcase #1 Not cited.%
	\or $\specialuparrow$#2%
	\else $\specialuparrow$#2%
	\fi%
}

\renewcommand{\epsilon}{\varepsilon}
\renewcommand{\ge}{\geqslant}
\renewcommand{\le}{\leqslant}
\renewcommand{\geq}{\geqslant}
\renewcommand{\leq}{\leqslant}
\renewcommand{\emptyset}{\varnothing}

\DeclarePairedDelimiter{\abs}{\lvert}{\rvert}
\DeclarePairedDelimiter{\ceil}{\lceil}{\rceil}
\DeclarePairedDelimiter{\floor}{\lfloor}{\rfloor}
\DeclarePairedDelimiter{\set}{\lbrace}{\rbrace}
\DeclarePairedDelimiter{\norm}{\lVert}{\rVert}

\DeclareMathOperator{\var}{Var}
\DeclareMathOperator{\lhs}{Left}
\DeclareMathOperator{\rhs}{Right}
\DeclareMathOperator{\pref}{Prefix}
\newcommand{\defn}[1]{\textcolor{Maroon}{\emph{#1}}}
\newcommand*{\eps}{\varepsilon}
\newcommand*{\cord}{\chi_{<}}

\newcommand*{\pord}{\pi_{<}}
\newcommand*{\rord}{\rho_{<}}

\newcommand*{\pairs}{\tau}
\newcommand*{\edged}{\alpha}

\newcommand*{\bE}{\mathbb{E}}
\newcommand*{\bN}{\mathbb{N}}

\newcommand*{\N}{\mathbb{N}}
\newcommand*{\bP}{\mathbb{P}}
\newcommand*{\bR}{\mathbb{R}}
\newcommand*{\bZ}{\mathbb{Z}}

\newcommand*{\cB}{\mathcal{B}}

\newcommand*{\cL}{\mathcal{L}}

\newcommand*{\cP}{\mathcal{P}}
\newcommand*{\cQ}{\mathcal{Q}}
\newcommand*{\cR}{\mathcal{R}}

\newcommand*{\puni}{\pi_{\textup{uni}}}
\newcommand*{\prich}{\pi_{\textup{rich}}}
\newcommand*{\ppal}{\pi_{\textup{pal}}}

\newcommand*{\prI}[1]{I_{#1}}

\newcommand*{\bit}{\set{0, 1}}

\usepackage[T1]{fontenc}
\usepackage[utf8]{inputenc}
\usepackage[UKenglish]{babel}
\usepackage[UKenglish]{isodate}

\title{Palettes determine relative Tur\'{a}n densities for ordered graphs}

\date{\today}

\begin{document}

\author{Domagoj Brada\v{c}\footnotemark[4] \qquad Lior Gishboliner\footnotemark[3] \qquad Freddie Illingworth\footnotemark[2] \\ Oliver Janzer\footnotemark[4] \qquad Arjun Ranganathan\footnotemark[2] \qquad Leo Versteegen\footnotemark[1] \\ Ella Williams\footnotemark[2]}

\maketitle

\begin{abstract}
    Reiher, R\"{o}dl, Sales, and Schacht initiated the study of relative Tur\'{a}n densities of ordered graphs and showed that it is more subtle and interesting than the unordered case. For an ordered graph $F$, its relative Tur\'{a}n density, $\rord(F)$, is the greatest $\edged$ such that every ordered graph $G$ has an $F$-free subgraph with at least $\edged e(G)$ edges. 

    We show that $\rord(F)$ is determined by a minimal palette construction for $F$ and give a finite algorithm for computing it exactly.
    This has many applications:
    in particular we resolve two problems of Reiher, R\"{o}dl, Sales, and Schacht and three problems of King et al.
    Moreover, in the course of the proof we show that $\rord(F)$ is equal to another natural parameter, namely the uniform Tur\'an density of $F$, which is the infimum $\alpha$ such that every ordered graph with density at least $\alpha$ between any two large intervals contains a copy of $F$. 
\end{abstract}

\renewcommand{\thefootnote}{\fnsymbol{footnote}} 

\footnotetext[0]{\emph{2020 MSC}: 05C35 (Extremal problems in graph theory)}

\footnotetext[2]{Department of Mathematics, University College London, UK (\textsf{\{\href{mailto:f.illingworth@ucl.ac.uk}{f.illingworth},\href{mailto:arjun.ranganathan.24@ucl.ac.uk}{arjun.ranganathan.24},\href{mailto:ella.williams.23@ucl.ac.uk}{ella.\allowbreak williams.23}\}@ucl.ac.uk}). Research of FI supported by the Heilbronn Institute for Mathematical Research. Research of EW supported by the Martingale Foundation.}

\footnotetext[1]{No affiliation (\textsf{\href{mailto:lversteegen.math@gmail.com}{lversteegen.math@gmail.com}}). Research carried out while at Department of Mathematics, The London School of Economics, UK and at the Mathematics Institute, University of Warwick, UK.}

\footnotetext[3]{Department of Mathematics, University of Toronto, Canada
(\textsf{\href{mailto:lior.gishboliner@utoronto.ca}{lior.gishboliner@utoronto.ca}}).
Supported in part by an NSERC Discovery Grant.}

\footnotetext[4]{Institute of Mathematics, EPFL, Lausanne, Switzerland (\textsf{\{\href{mailto:domagoj.bradac@epfl.ch}{domagoj.bradac},\href{mailto:oliver.janzer@epfl.ch}{oliver.janzer}\}@epfl.ch}).}

\renewcommand{\thefootnote}{\arabic{footnote}} 

\section{Introduction}

Tur\'{a}n problems are some of the oldest and most fundamental in graph theory. The archetypal question asks, given a graph $F$ and a host graph $G$, for the greatest edge density of an $F$-free subgraph of $G$. That is, for the value of
\begin{equation*}
    \rho(F, G) \coloneqq \max\set[\bigg]{\frac{e(G')}{e(G)} \colon G' \subseteq G, F \nsubseteq G'}.
\end{equation*}
The classical Tur\'{a}n problem, solved by the theorems of Mantel~\cite{Mantel1907}, Tur\'{a}n~\cite{Turan1941}, and Erd\H{o}s, Stone, and Simonovits~\cite{ErdosStone1946,ErdosSimonovits1966}, addresses the case where the host graph $G$ is complete. Together they show that the \defn{Tur\'{a}n density} of a graph $F$ is
\begin{equation}\label{eq:ESS}
    \pi(F) \coloneqq \lim_{n \to \infty} \rho(F, K_n) = 1 - \frac{1}{\chi(F) - 1}.
\end{equation}
Tur\'{a}n problems on non-complete host graphs have been widely studied (see for example \cite{Chung1992,SzaboVu2003,KRS2004,AKS2007,FurediOzkahya2011,ConlonGowers2016,Schacht2016}) although the problems are typically much more difficult. For example, Erd\H{o}s's \$100 problem~\cite{Erdos1984,Erdos1990} from 1984 of determining $\lim_{d \to \infty} \rho(C_4, Q_d)$ remains open to this day~\cite{Bloom}.

Given the interest in this class of problems, it is natural to ask for the infimum of $\rho(F, G)$ over all host graphs $G$, which is called the \defn{relative Tur\'{a}n density} of $F$ and denoted $\rho(F)$. That is, to ask for the greatest $\edged \in [0, 1]$ such that every graph $G$ has an $F$-free subgraph on at least $\edged e(G)$ edges. However, this infimum is always given by $\pi(F)$ and so complete graphs are universally hardest to make $F$-free.\footnote{The upper bound $\inf_G \rho(F, G) \leq 1 - 1/(\chi(F) - 1)$ follows from \eqref{eq:ESS}. To see that $\inf_G \rho(F, G) \geq 1 - 1/(\chi(F) - 1)$, colour the vertices of $G$ randomly with $\chi(F) - 1$ colours and keep only the edges between vertices of distinct colours.}

Reiher, R\"{o}dl, Sales, and Schacht~\cite{relative-ordered-turan-origin} proved that the situation is more interesting and nuanced for ordered graphs, which are graphs whose vertex set is equipped with a total ordering.  
For a survey on extremal problems in ordered graphs see Tardos's ICM talk~\cite{Tardos2018} or its extended version~\cite{Tardos2019}.
Analogously to graphs, the Tur\'{a}n problem for ordered graphs asks, given an ordered graph $F$ and an ordered host graph $G$, for \defn{$\rord(F, G)$}, the greatest density of an $F$-free subgraph of $G$. When the host graph is complete, Pach and Tardos~\cite{PachTardos} answered the problem to leading order by proving the ordered analogue of the Erd\H{o}s-Stone-Simonovits theorem:
\begin{equation*}
    \pord(F) \coloneqq \lim_{n \to \infty} \rord(F, K_n) = 1 - \frac{1}{\cord(F) - 1},
\end{equation*}
where $\cord(F)$ is the \defn{ordered}/\defn{interval chromatic number}. While it remains true that the relative Tur\'{a}n density, \defn{$\rord(F)$}, of an ordered graph satisfies
\begin{equation}\label{eq:pi-upper}
    \rord(F) \coloneqq \inf_G \rord(F, G) \leq \inf_n \rord(F, K_n) = 1 - \tfrac{1}{\cord(F) - 1},
\end{equation}
Reiher, R\"{o}dl, Sales, and Schacht \cite{relative-ordered-turan-origin} showed that equality does not hold in general. They did so by proving that the monotone path of length $k$, $P_{k + 1}^{\mathrm{mon}}$, satisfies $\rord(P_{k + 1}^{\mathrm{mon}}) = 1/2 \cdot (1 - 1/k)$
while 
$\pord(P_{k + 1}^{\mathrm{mon}}) = 1 - 1/k$ is twice as large. 
The lower bound construction showing that 
$\rord(P_{k + 1}^{\mathrm{mon}}) \geq 
1/2 \cdot (1 - 1/k)$ is as follows: Given an ordered graph $G$, take a random colouring $f \colon V(G) \to [k]$, and only keep edges with increasing labels, i.e., edges $xy$ with $x<y$ and $f(x) < f(y)$. It is easy to see that the resulting subgraph of $G$ is 
$P_{k + 1}^{\mathrm{mon}}$-free, and that the expected proportion of edges kept is 
$1/2 \cdot (1 - 1/k)$.

The construction described in the preceding paragraph is an example of a \defn{palette construction}. Our main result, stated as \cref{thm:palette} below, shows that such constructions determine the relative Tur\'an density $\rord(F)$ \emph{for every ordered graph} $F$. Let us now define the general construction.
For two directed graphs $H$ and $D$, we write $H \to D$ if there exists a (not necessarily injective) mapping $\psi \colon V(H)\to V(D)$ such that if $(u, v) \in A(H)$ then $(\psi(u), \psi(v)) \in A(D)$. We write $H \nrightarrow D$ if no such mapping exists.
For an ordered graph $F$, the \defn{forward orientation} of $F$, denoted \defn{$\vec{F}$}, is the oriented graph on vertex set $V(F)$ such that there is an edge from $i$ to $j$ if and only if $i<j$ and $ij\in E(F)$. Given a digraph $D$, we say that \defn{$F$ is embeddable into $D$} if $\vec{F}\rightarrow D$. We say that $D$ is a \defn{palette} for $F$ if $F$ is \emph{not} embeddable into $D$.

The \defn{Lagrangian polynomial} of a digraph $D$ is
\[
    \cL(D, \boldsymbol{x}) \coloneqq \sum_{(v, v') \in A(D)} x_{v} x_{v'},
\]
where $\boldsymbol{x} = (x_v \colon v \in V(D))$.
Observe that this does not depend on the directions of the arcs of $D$, but only on the number of arcs between each pair of vertices (which for a palette is 0, 1, or 2). The \defn{Lagrangian} of $D$ is
\[
    \cL(D) \coloneqq \max\set[\Big]{\cL(D, \boldsymbol{x}) \colon x_v \geq 0, \sum_{v \in V(D)} x_v = 1}.
\]
Finally, we define the palette density of $F$.
\begin{definition}[palette density]
    The \defn{palette density} of an ordered graph $F$ is
    \begin{equation}\label{eq:palette density}
        \ppal(F) \coloneqq \sup\set{\cL(D) : D \text{ is a palette for } F }.
    \end{equation}
\end{definition}
It is easy to see that $\ppal(F)$ is a lower bound for $\rord(F)$. Indeed, let $D$ be a palette for $F$, meaning that $F$ is not embeddable into $D$.
Let $\boldsymbol{x} = (x_v \colon v \in V(D))$ be a maximiser in the definition of $\cL(D)$. Let $G$ be any ordered graph. Construct a random map $f \colon V(G) \to V(D)$ where $\bP[f(u) = v] = x_v$, independently for each $u \in V(G)$. Let $G'$ be the spanning subgraph of $G$ consisting of all edges $u u'$ where $u < u'$ and 
$(f(u), f(u')) \in A(D)$. The probability that a fixed edge $u u'$ ($u < u'$) of $G$ is kept in $G'$ is 
\[
\sum_{(v, v') \in A(D)} \bP[f(u) = v] \cdot \bP[f(u') = v'] = 
\sum_{(v, v') \in A(D)} x_{v} x_{v'} = \cL(D).
\]
Thus, by averaging, there is some choice of $f$ where $e(G') \geq \cL(D) \cdot e(G)$. Note that $G'$ is embeddable into $D$ and so must be $F$-free. Thus $\rord(F, G) \geq \cL(D)$. Since this is valid for all $G$ and $D$, we have $\rord(F) \geq \ppal(F)$. 
Our main result is that equality holds.
\begin{theorem}[palettes determine relative Tur\'an density]\label{thm:palette}
    For every ordered graph $F$, $\rord(F) = \ppal(F)$.
\end{theorem}

\Cref{thm:palette} provides a characterisation of the relative Tur\'an density of ordered graphs that is more amenable to computation. We note that a similar characterisation was proved for the uniform Tur\'{a}n density of hypergraphs by Lamaison~\cite{lamaison2024palettesdetermineuniformturan}.

\subsection{Applications of \texorpdfstring{\cref{thm:palette}}{palette theorem}}

\Cref{thm:palette} has several interesting applications. A natural question is to determine which ordered graphs $F$ have relative Tur\'an density $\rord(F) = 0$. By \cref{thm:palette}, this is equivalent to having $\ppal(F) = 0$. Observe that this is in turn equivalent to there being a homomorphism from $\vec{F}$ to every digraph $D$ with at least one edge. Thus, 
$\vec{F} \to \vec{K_2}$. Such a homomorphism is a colouring $\psi$ of $V(F)$ with two colours 1,2, such that every edge $xy \in E(F)$ with $x<y$ satisfies $\psi(x) = 1$, $\psi(y) = 2$. It is easy to see that such a colouring exists if and only if 
$F$ has no monotone path of length 2. Thus, we obtain the following.
\begin{theorem}\label{thm:vanishing}
    For every ordered graph $F$, $\rord(F) = 0$ if and only if $F$ does not contain a monotone path of length two.
\end{theorem}
This answers a question of King, Lidick\'{y}, Ouyang, Pfender, Wang, and Xiang \cite{KLOPWX}, confirming their prediction.

We will show that computing $\ppal(F)$ is a finite problem, and that the supremum in \eqref{eq:palette density} is attained. Furthermore, this value is rational.
\begin{theorem}\label{thm:rho rational}
    $\rord(F)$ is rational for every ordered graph $F$. Furthermore, there is an algorithm for computing $\rord(F)$.
\end{theorem}

For an ordered graph $F$ with $\chi(F) = 2$, we can determine $\rord(F)$ exactly, as follows. 
For an oriented graph $H$ with no directed cycle, let \defn{$t(H)$} be the minimum $t$ such that $H$ is homomorphic to every tournament on $t$ vertices.

\begin{theorem}\label{thm:bipartite F}
    For every ordered graph $F$ with $\chi(F) = 2$,
    \[
    \rord(F) = 
    \frac{1}{2}\biggl( 
    1 - \frac{1}{t(\vec{F})-1} \biggr).
    \]
\end{theorem}
\begin{proof}
    By \cref{thm:palette}, we need to show that $\ppal(F) = 
    \frac{1}{2}\bigl( 
    1 - \frac{1}{t-1} \bigr)$, where
    $t \coloneqq t(\vec{F})$. 
    For the lower bound, take $D$ to be a tournament on $t - 1$ vertices with $\vec{F} \nrightarrow D$. 
    Then
    $\ppal(F) \geq \cL(D) = \frac{1}{2}\bigl( 
    1 - \frac{1}{t-1} \bigr)$.
    For the upper bound, 
    let $D$ be any digraph with 
    $\cL(D) > 
    \frac{1}{2}\bigl( 
    1 - \frac{1}{t-1} \bigr)$, and let us show that $\vec{F} \rightarrow D$. 
    This is true if $D$ contains a bidirectional edge, because $F$ is bipartite by assumption. Else, $D$ is an oriented graph. 
    Then, by the weighted version of Tur\'an's theorem (see, e.g., \cite{motzkin1965maxima}), $D$ contains a tournament of size $t = t(\vec{F})$. 
    By the definition of $t(\vec{F})$, we have $\vec{F} \rightarrow D$, as required. 
\end{proof}

\Cref{thm:bipartite F} allows us to determine $\rord(F)$ for every ordered forest $F$. For a digraph $H$, let \defn{$\ell(H)$} denote the maximum length of a directed path in $H$. 

\begin{theorem}\label{thm:tree}
    For every ordered forest $F$ with at least one edge, 
    \[
    \rord(F) = 
    \frac{1}{2}
    \biggl(
    1 - \frac{1}{\ell(\vec{F})}
    \biggr).
    \]
\end{theorem}

\begin{proof}
    By \cref{thm:bipartite F}, we only need to show that 
    $t(\vec{F}) = \ell+1$, where
    $\ell \coloneqq \ell(\vec{F})$. First, 
    note that the homomorphic image of a directed path is a directed walk of the same length. A transitive tournament on $\ell$ vertices has no directed walk of length $\ell$, and hence no homomorphic image of $\vec{F}$.
    For the other direction, let $T$ be a tournament on $\ell+1$ vertices, and let us show that $\vec{F} \to T$. 
    Every orientation of a forest is homomorphic to the cyclic triangle,\footnote{For edgeless forests this is immediate, otherwise it can be proved by induction on the number of vertices, by deleting a leaf $v$. Let $u$ be the only neighbour of $v$. Since $\psi(u)$ has both an out-neighbour and an in-neighbour on the cyclic triangle, there is always a suitable image for $v$.} and so we may assume that $T$ is transitive. 
    To see that $\vec{F} \to T$, number the vertices of $T$ as 
    $0, \dots, \ell$ (with edges directed from left to right) and map each $v \in V(F)$ to the vertex whose number is the length of the longest directed path finishing at $v$ in $\vec{F}$. It is easy to see that this is a homomorphism, as required.
\end{proof}

\Cref{thm:tree} in particular implies the aforementioned result of 
Reiher, R\"{o}dl, Sales, and Schacht \cite{relative-ordered-turan-origin} that 
$\rord(P_{k + 1}^{\mathrm{mon}}) = 1/2 \cdot (1 - 1/k)$.

Let $C_{k}^{\mathrm{mon}}$ denote the monotone $k$-cycle, i.e., the cycle $v_1v_2\dots v_kv_1$ with the ordering 
$v_1 < \dots < v_k$. 
Answering a question of Reiher, R\"{o}dl, Sales, and Schacht \cite{relative-ordered-turan-origin}, we determine 
$\rord(C_{k}^{\mathrm{mon}})$.
Consider first the case that $k$ is even, where \cref{thm:bipartite F} applies.

\begin{theorem}\label{thm:even cycle}
    For every $k \geq 2$, 
    $\rord(C_{2k}^{\mathrm{mon}}) = 
    \tfrac{1}{2}\bigl(1 - 
    \tfrac{1}{2k-1}\bigr)$.
\end{theorem}
\begin{proof}
    By \Cref{thm:bipartite F}, we only need to show that 
    $t(\vec{C}) = 2k$, where
    $\vec{C} \coloneqq \vec{C}_{2k}^{\mathrm{mon}}$.
    For the lower bound, a transitive tournament on $2k-1$ vertices has no homomorphic image of $\vec{C}$ (because $\ell(\vec{C}) = 2k-1$). For the upper bound, it suffices to show that for every $k \geq 1$, every tournament on $2k$ vertices has a directed Hamilton path 
    $v_1 \to v_2 \to \dots \to v_{2k}$ with $v_1 \to v_{2k}$. This was proved in \cite{tournament_cycle}.\footnote{In fact, they proved that the same conclusion holds for every number of vertices except 3 and 5.}
\end{proof}

Next, we consider odd cycles, showing that they have relative Tur\'an density 
$1/2$. We in fact prove the following more general statement, determining the relative Tur\'an density of every 3-chromatic ordered graph with a critical edge.

\begin{theorem}\label{thm:edge-critical}
    Let $F$ be an ordered graph such that $\chi(F) = 3$ and there is some edge 
    $e \in E(F)$ with $\chi(F - e) = 2$. Then $\rord(F) = \frac{1}{2}$.
\end{theorem}

For the proof of \cref{thm:edge-critical}, it is convenient to use some properties of minimal palettes (see \cref{sec:minimal palettes}). We hence postpone the proof to that section.

Using \cref{thm:bipartite F,thm:edge-critical}, we can determine the relative Tur\'an density of the following ordered graph, answering a question of King et al.~\cite{KLOPWX}. Let $Q_{a,b}$ be the ordered graph obtained from the monotone path with vertices $1,2,\dots,a+b+1$ by adding the edge 
$\set{1,a+1}$. If $a$ is even, then $Q_{a,b}$ is 3-chromatic and has a critical edge, so $\rord(Q_{a,b}) = \frac{1}{2}$. If $a$ is odd, then $Q_{a,b}$ is bipartite, and \cref{thm:bipartite F} gives $\rord(Q_{a,b}) = 
\frac{1}{2}\bigl( 1 - \frac{1}{a+b} \bigr)$. 
Indeed, similarly to the proof of  \cref{thm:even cycle}, for the upper bound we only need the fact that every tournament on $1+a+b$ vertices contains a directed path $v_1 \to \dots \to v_{a+b+1}$ with the edge 
$v_1 \rightarrow v_{a+1}$, assuming that $a$ is odd. 
This was proved in \cite{petrovic1988some}. 

Our next application characterises when the relative Tur\'{a}n density equals the Tur\'{a}n density, thus answering another question of 
Reiher, R\"{o}dl, Sales, and Schacht \cite{relative-ordered-turan-origin}.

\begin{restatable}{theorem}{thmRhoChromaticEquiv}
\label{thm:rho_chromatic_equiv}
    For every ordered graph $F$, 
    $\rord(F) = 1- \frac{1}{\chi_<(F)-1}$ if and only if $\chi_<(F) = \chi(F)$.
\end{restatable}

Finally, \cref{thm:palette} implies that if $F'$ consists of the disjoint union of two copies of $F$, where the vertex sets of the two copies are allowed to interlace in any way, then $\rord(F') = \rord(F)$. Indeed, it is easy to see that $\ppal(F') = \ppal(F)$. This answers another question of King et al.~\cite{KLOPWX}.

One key question that our work leaves open is to determine the set of values attained by the relative Tur\'an density $\rord(\cdot)$.

\begin{problem}\label{prob:spectrum}
    For which $\rho \in [0,1]$ does there exist an ordered graph $F$ with $\rord(F) = \rho$?
\end{problem}

We can resolve \cref{prob:spectrum} for densities below $\frac{1}{2}$: the values of $\rord(\cdot)$ below 
$\frac{1}{2}$ are precisely the numbers of the form $\frac{1}{2}(1 - \frac{1}{k})$, $k \geq 1$. Indeed, if $\rord(F) < \frac{1}{2}$ then $F$ is bipartite (because if $F$ is not bipartite then the bidirectional edge is a palette for $F$, and hence 
$\rord(F) \geq \frac{1}{2}$). \cref{thm:bipartite F} shows that for a bipartite $F$, the relative Tur\'an density $\rord(F)$ always has the form 
$\frac{1}{2}(1 - \frac{1}{k})$, and the fact that $\rord(P_{k + 1}^{\mathrm{mon}}) = \tfrac{1}{2}\bigl(1 - \tfrac{1}{k}\bigr)$ shows that every such value is attained. 

In the next regime, we can show that 
the set of attained values in the interval $[\frac{1}{2}, \frac{2}{3})$ is 
$\set[\big]{\tfrac{2}{3} \cdot \tfrac{3a}{3a + 1} \colon a \in \bN}$ as well as possibly $\frac{7}{12}$.
In particular, $\frac{2}{3}$ is the second smallest accumulation point (after $\frac{1}{2}$). 
We can also show that $\frac{7}{10}$ is the third smallest accumulation point.
We omit these proofs.

\subsection{Uniform Tur\'an density and the universal construction}
The proof of \cref{thm:palette} has two intermediate steps which are interesting in their own right. 
First, recall that in the unordered setting, the complete graphs are universally optimal host graphs for the Tur\'an problem (meaning that $\rho(F) = \lim_{n \rightarrow \infty} \rho(F,K_n)$), but this is far from true for ordered graphs. 
Nevertheless, perhaps surprisingly, there does exist a family of universal host graphs 
for $\rord$. This family, denoted $R(m,d)$, was introduced in \cite{Gmd-origin}. Its definition is given in \cref{subsec:Rmd}. This universality is stated in the following theorem.

\begin{restatable}{theorem}{universal}\label{thm:universal}
    For every ordered graph $F$, 
    \[
    \rord(F) = \lim_{d \to \infty} \limsup_{m \to \infty} \rord(F, R(m, d)).
    \]
\end{restatable}

This theorem reduces determining $\rord(F)$ to solving the Tur\'{a}n problem on the specific host graph $R(m, d)$. 

The next step in the proof of \cref{thm:palette} is to show that $\rord(F)$ is equal to another natural density parameter, defined as follows.
\begin{definition}\label{def:interval-dense}
    An $n$-vertex ordered graph $G$ is called 
    \defn{$(\alpha,\delta)$-interval-dense} if $e(I, J) \geq \alpha \abs{I}\abs{J}$ for every pair of disjoint intervals $I,J$ in $G$ with 
    $\abs{I}, \abs{J} \geq \delta n$. 
\end{definition}

\begin{definition}[uniform Tur\'an density]
The \defn{uniform Tur\'an density} of an ordered graph $F$, denoted by \defn{$\puni(F)$}, is defined as the infimum over all $\alpha$ for which there exist $\delta>0$ and $n_0$ such that every $(\alpha,\delta)$-interval-dense ordered graph on at least $n_0$ vertices contains a copy of $F$.  
\end{definition}

An analogous quantity has been well studied for (unordered) hypergraphs (see for example \cite{erdHos1982ramsey, glebov2016problem, reiher2018turan, bucic2023uniform, lamaison2024palettesdetermineuniformturan, lamaison2025uniform}). Our next result is that the relative Tur\'an density and the uniform Tur\'an density coincide for ordered graphs.

We note that instead of requiring density at least $\alpha$ between any disjoint pair of large intervals, one can require density at least $\alpha$ within each large interval. Namely, 
it is possible to show that the uniform Tur\'an density $\puni(F)$ has the following natural alternative definition: It is the infimum $\alpha$ such that an ordered graph $G$ contains $F$ whenever each interval $I$ of $V(G)$ with $\abs{I} \geq \delta n$ satisfies $e(G[I]) \geq \alpha \binom{\abs{I}}{2}$, where
$n = \abs{V(G)} \geq n_0$ and $\delta,n_0$ depend on $\alpha$ and $F$. We omit the proof of this equivalence.

\begin{theorem}\label{thm:rho_uni_equiv}
    For every ordered graph $F$, we have $\rord(F) = \puni(F)$.
\end{theorem}

\paragraph{Paper and proof organisation.}
In order to prove 
\cref{thm:palette,thm:universal,thm:rho_uni_equiv}, we will prove the following chain of inequalities:
\begin{equation}\label{eq:inequality chain}
\lim_{d \to \infty} \limsup_{m \to \infty} \rord(F, R(m, d)) \leq
\puni(F) \leq \ppal(F).
\end{equation}
Clearly, 
$\rord(F) \leq \rord(F, R(m, d))$ for all $d,m$. Also, we saw above that 
$\rord(F) \geq \ppal(F)$. Thus, the inequalities in \eqref{eq:inequality chain} would imply that all of these quantities are equal, as required.
The left inequality in \eqref{eq:inequality chain} is proved in  
\cref{sec:universality}-\cref{sec:uni-rich} (see \cref{lemma:reduction,cor:rich<uni}), building on definitions from \cref{sec:definitions}.
To that end, we introduce an intermediate density parameter, denoted $\prich$, and show that 
$\lim_{d \to \infty} \limsup_{m \to \infty} \rord(F, R(m, d)) \leq \prich(F) \leq
\puni(F)$.
The right inequality in \eqref{eq:inequality chain} is proved in \cref{sec:pal-uni} (see \cref{thm:pal >= uni}). 

In \cref{sec:minimal palettes} we establish some properties of \defn{minimal digraphs}, i.e., digraphs $D$ such that 
$\cL(D') < \cL(D)$ for every proper subdigraph $D'$ of $D$ with at least one vertex. We then use these to prove \cref{thm:rho rational,thm:edge-critical,thm:rho_chromatic_equiv}.

\paragraph{AI Declaration.} All results were obtained without the use of AI. The presented proofs are entirely due to the authors with the exception of the proof of \cref{thm:rho_chromatic_equiv}. We previously had a proof of this theorem using uniform Tur\'{a}n density. ChatGPT suggested the current, simpler proof. ChatGPT was also used for a final proofreading of the paper.

\section{Richness and the universal construction}\label{sec:definitions}

In this section we define the construction $R(m,d)$ appearing in \cref{thm:universal}, as well as introduce the density parameter $\prich$ which will play a key role in the proof of \eqref{eq:inequality chain}. We begin with some definitions. 

Let $G$ be an ordered graph. Given $u\in V(G)$, we say an edge $uv$ is a \defn{forward edge} for $u$ if $u < v$, and $v$ is a \defn{forward neighbour} of $u$ in this case. We define \defn{backward edges} and \defn{backward neighbours} analogously, if $v < u$. 
We denote by $\boldsymbol{1}$ the vector of all 1s.

\subsection{Binary ordering}

We will often consider ordered graphs whose vertices are labelled by binary strings of the form $x\in \bit^d$, indexed as $x = x_1x_2\dots x_d$.
For distinct binary strings $x, y \in \bit^d$, we define
\begin{equation*}
    \delta(x, y) \coloneqq \min\set{i \colon x_i \neq y_i},
\end{equation*}
and so, for example, $\delta(001110, 001101) = 5$.
We will always view $\bit^d$ as being ordered lexicographically.

\begin{definition}[lexicographic ordering]
    For distinct $x, y \in \bit^d$, define \defn{$x < y$} if $x_{\delta(x, y)} = 0$ and $y_{\delta(x, y)} = 1$.
\end{definition}

If one interprets each string as a binary number (so $x$ corresponds to $\sum_i x_i 2^{d - i}$), then this corresponds to the usual order on the non-negative integers.
The set $\bit^d$ with its lexicographic ordering splits naturally into intervals.

\begin{definition}[fundamental intervals]
    For $\ell \in \set{0,1,\dots,d}$, let \defn{$\cP_\ell$} be the family of equivalence classes of the relation on $\bit^d$ given by $\delta(u,v)> \ell$. Each $I \in \cP_\ell$ is an interval in $\bit^d$ consisting of strings that agree on the first $\ell$ bits. We call these intervals the \defn{fundamental intervals at level $\ell$}. Given $I\in \cP_\ell$ ($\ell < d$), we denote the half of $I$ in $\cP_{\ell+1}$ for which the $(\ell+1)$th digit is $0$ by \defn{$\lhs(I)$}, and the other half by \defn{$\rhs(I)$}.
\end{definition}

Note that $\abs{\cP_\ell} = 2^\ell$ and each interval $I \in \cP_\ell$ has size $2^{d - \ell}$. Observe that $\cP_b$ refines\footnote{A partition $\cP$ \defn{refines} a partition $\cQ$ if, for all $P \in \cP$, there is some $Q \in \cQ$ such that $P \subseteq Q$.} $\cP_a$ for $a \leq b$. For $a \leq b$ and an interval 
$I \in \cP_a$, we let $\cP_b(I)$ denote the set of all $J \in \cP_b$ with $J \subseteq I$. 

\subsection{Graphs on \texorpdfstring{$\bit^d$}{{0,1}d} and richness}\label{subsec:rich}

Our proofs of \cref{thm:universal,thm:rho_uni_equiv} go via an intermediate notion, richness, and intermediate density, $\prich$, which we show is equal to the relative Tur\'{a}n density. We now define richness and $\prich$.

Let $G$ be an ordered graph whose vertex set is $\bit^d$, ordered lexicographically.
We say that $uv \in E(G)$ is a \defn{level-$\ell$ edge} if $\delta(u,v) = \ell$. 
We denote the set of level-$\ell$ edges of $G$ by \defn{$E_{\ell}(G)$} and their number by \defn{$e_{\ell}(G)$}. When $A$ is a subset of $\bit^d$ we write \defn{$E_{\ell, G}(A)$} and \defn{$e_{\ell, G}(A)$} to mean $E_{\ell}(G[A])$ and $e_{\ell}(G[A])$, respectively, where $G[A]$ is the induced subgraph of $G$ with vertex set $A$. We may omit $G$ from the subscripts when the context is clear. 
To help us normalise the number $e_{\ell, G}(A)$ correctly, we define \defn{$\pairs_\ell(A)$} as
\begin{equation*}
    \pairs_{\ell}(A) \coloneqq \abs{\set{(u, v) \in A \times A \colon u < v, \ \delta(u, v) = \ell}},
\end{equation*}
the total number of pairs $\set{u, v} \subseteq A$ for which $\delta(u,v) = \ell$ (note that this is independent of $G$). 
We write
\begin{equation*}
    \pairs_{\ell, d} \coloneqq \pairs_\ell(\bit^d) = 2^{2d - \ell - 1}.
\end{equation*}

We are particularly interested in levels on which $G$ has a large proportion of all possible edges.

\begin{definition}[rich levels and graphs]\label{def:rich}
    For a graph $G$ on $\bit^d$ and $\edged \in \bR$, we say that a level $\ell \in [d]$ is \defn{$\edged$-rich with respect to $G$} if $e_{\ell}(G) \geq \edged\pairs_{\ell,d}$. We say that $G$ is \defn{$(\alpha, C)$-rich} for some $C>0$ if it contains at least $C$ levels that are $\alpha$-rich. 
\end{definition}

The rich ordered Tur\'{a}n density, denoted \defn{$\prich$}, is defined as
\begin{equation*}
    \prich(F) \coloneqq \inf\set{\edged \in [0, 1] \colon \exists C \textup{ such that every $(\edged, C)$-rich graph contains } F}.
\end{equation*}
Note that richness is only defined for graphs on $\bit^d$ (for some $d$).
To aid intuition we remark that if $G$ is an $(\edged, \delta)$-interval-dense graph on $\bit^d$ (see \cref{def:interval-dense}), then the first $\min\set{\floor{\log_2(1/\delta)}, d}$ levels are $\edged$-rich and so $G$ is $(\alpha, \min\set{\floor{\log_2(1/\delta)}, d})$-rich.

\subsection{\texorpdfstring{$R(m, d)$}{R(m,d)}}\label{subsec:Rmd}

We will also need to discuss ordered graphs on $\bit^d \times [m]$. Their vertex set is ordered lexicographically: $(x, i) < (y, j)$ if $x < y$, or $x = y$ and $i < j$. For each $x \in \bit^d$, we define $B_x = \set{x} \times [m]$ and refer to the $B_x$ as \defn{blocks}. Just as for graphs on $\bit^d$, we say an edge $uv$ is a \defn{level-$\ell$ edge} if $u \in B_x$ and $v \in B_y$ where $\delta(x, y) = \ell$ and we write $e_\ell(G)$ for the number of level-$\ell$ edges in a graph $G$ on $\bit^d \times [m]$.

The construction $R(m,d)$ is obtained by placing a random bipartite graph with density $2^{-d+\delta(x,y)}$ between $B_x,B_y$ for every pair of distinct 
$x,y \in \bit^d$. The following proposition states typical properties of this random construction.

\begin{proposition}\label{prop:Rmd-properties}
    There exists a family of graphs \defn{$R(m, d)$} \textup{($d, m \in \bN$)} such that $R(m, d)$ has vertex-set $\bit^d \times [m]$ and, for every $d \in \N$, $\eps > 0$ and for all sufficiently large $m$\textup{:}
    \begin{enumerate}[label = \upshape{(\roman{*})}]
        \item for every $\ell \in [d]$, $e_\ell(R(m, d)) = (1 \pm \epsilon) 2^{d - 1} m^2$ and so $e(R(m, d)) = (1 \pm \eps) d 2^{d - 1} m^2$\textup{;}
        \item for every distinct $x, y \in \bit^d$ and every $P \subseteq B_x$ and $Q \subseteq B_y$, each of size at least $m^{2/3}$, the number of edges in $R(m, d)$ between $P$ and $Q$ is at most $(1 + \eps) 2^{-d + \delta(x, y)} \abs{P} \abs{Q}$.
    \end{enumerate}
\end{proposition}
\begin{proof}
    For all $d, m\in \N$, consider the random graph $\cR(m,d)$ on $\bit^d\times [m]$ where an edge $(x,i)(y,j)$ (for $x < y$) is present with probability $2^{-d+\delta(x,y)}$ independently of all other edges. Two simple applications of Chernoff's inequality show that, for each $d \in \N$ and sufficiently large $m$, with positive probability $\cR(m, d)$ has the properties above with $\eps = m^{-1/100}$. 
    Take $R(m, d)$ to be such a sample of $\cR(m, d)$ when $m$ is sufficiently large compared to $d$ and otherwise pick $R(m, d)$ arbitrarily.
\end{proof}

Here is one useful corollary of the first property in \cref{prop:Rmd-properties}, relating to richness. Note that the denominator $2^{d - 1} m^2$ is roughly $e_\ell(R(m, d))$.

\begin{proposition}\label{prop:Rmd-richness}
    Let $d\in \N$ and $\eps > 0$. Then, for all sufficiently large $m$ and every $\edged \in [0, 1]$, every subgraph $G$ of $R(m, d)$ with at least $(\edged + 2 \eps) \cdot e(R(m, d))$ edges satisfies
    \begin{equation*}
        \tfrac{1}{d} \sum_{\ell = 1}^d \frac{e_{\ell}(G)}{2^{d - 1} m^2} \geq \edged + \epsilon.
    \end{equation*}
\end{proposition}

\begin{proof}
For each level $\ell\in [d]$, we define $p_\ell = e_\ell(G)/(2^{d - 1} m^2)$. Now, \cref{prop:Rmd-properties}(i) with input $\eps/2$ implies
    \begin{equation*}
        \tfrac{1}{d} \sum_\ell p_\ell = \frac{1}{d \cdot 2^{d - 1} m^2} \cdot \sum_\ell e_\ell(G) 
        =\frac{e(G)}{e(R(m, d))} \frac{e(R(m, d))}{d2^{d - 1} m^2}
        \geq (\edged + 2 \eps)(1 - \eps/2) \geq \edged + \eps. \qedhere
    \end{equation*}
\end{proof}

\section{Reduction to richness}\label{sec:universality}

The purpose of this section is to prove \cref{lemma:reduction}, which reduces obtaining upper bounds on $\rord(F)$ to finding $F$ in all graphs with sufficiently large `average richness'. In essence, this is a regularity reduction from the graph $R(m,d)$ (whose vertex-set is $\bit^d \times [m]$) to graphs on $\bit^d$.
To that end, we will use Szemer\'{e}di's regularity lemma~\cite{Szemeredi1978regularity} together with some associated machinery.
For a pair of disjoint vertex sets $X, Y \subseteq V(G)$, we use $d(X, Y) = e(X, Y) \abs{X}^{-1} \abs{Y}^{-1}$ to denote the \defn{density} between $X$ and $Y$.

For $\eta > 0$, we say a pair $(A,B)$ of disjoint vertex sets is \defn{$\eta$-regular} if for all $X \subseteq A$ and $Y \subseteq B$ satisfying $\abs{X} \geq \eta \abs{A}$ and $\abs{Y} \geq \eta \abs{B}$, we have 
\begin{equation*}
    \abs{d(X,Y) - d(A,B)} \leq \eta.
\end{equation*}
A partition $\cQ = \set{V_1, \dots, V_t}$ of $V(G)$ is an \defn{equipartition} if $\abs[\big]{\abs{V_i} - \abs{V_j}} \leq 1$ for all $i$ and $j$. This implies that $\floor{\abs{V(G)}/t} \leq \abs{V_i} \leq \ceil{\abs{V(G)}/t}$ for all $i$.
Finally, $\cQ$ is \defn{$\eta$-regular} if $(V_i, V_j)$ is $\eta$-regular for all but at most $\eta \binom{t}{2}$ pairs $i < j$.

We need a version of the regularity lemma where the resulting $\eta$-regular partition refines some initial equipartition. The existence of such a partition follows from the proof of the usual regularity lemma (see, for example, \cite[Lemma~1.15]{KomlosSimonovits1996}).

\begin{lemma}[refined regularity lemma]\label{lemma:regularity}
    For every $\eta > 0$ and $r\in \N$, there exists $L\in \N$ such that the following holds. Let $G$ be a graph and $\cQ$ be an equipartition of $V(G)$ into $r$ parts. Then there exists an $\eta$-regular equipartition $\cP$ of $G$ that refines $\cQ$, and such that each part of $\cQ$ is refined into at most $L$ parts.
\end{lemma}

We will also need the following building lemma (see, for example, \cite[Theorem~2.1]{KomlosSimonovits1996}).

\begin{lemma}[graph building lemma]\label{lemma:graph_building}
    Let $H$ be a graph on the vertex-set $1,\dots,h$ and let 
    $\lambda \in (0, 1)$. For all sufficiently small $\eta > 0$ the following holds. Suppose $V_1, \dots, V_{h}$ are sufficiently large pairwise disjoint vertex sets such that $(V_i, V_j)$ is $\eta$-regular of density at least $\lambda$ for each $ij \in E(H)$. Then there is a copy of $H$ with each vertex in the corresponding $V_i$.
\end{lemma}

We are now ready to prove the main result of this section:

\begin{theorem}[reduction to $\bit^d$]\label{lemma:reduction}
For every ordered graph $F$,
    \[
    \limsup_{d \to \infty} \limsup_{m \to \infty} \rord(F, R(m, d)) \leq 
    \prich(F).
    \]
\end{theorem}

\begin{proof}
    Put $\edged \coloneqq \prich(F)$ and fix any $\varepsilon > 0$. 
    By the definition of $\prich$, there exists $C$ such that for every $d$, every $(\alpha+\varepsilon,C)$-rich graph on $\bit^d$ contains a copy of $F$. 
    Fix
    $m \gg d \geq C/\varepsilon$, and let us show that 
    $\rord(F,R) \leq \alpha + 16\varepsilon$, where $R \coloneqq R(m,d)$. 
    As $\varepsilon$ was arbitrary, this would show that 
    $\limsup_{d \to \infty} \limsup_{m \to \infty} \rord(F, R(m, d)) \leq \alpha$.
    
    So let $R' \subseteq R$ be a spanning subgraph of $R$ with 
    $e(R') \geq (\edged + 16\eps) \cdot e(R)$, and let us show that $R'$ contains a copy of $F$.  
    For each level $\ell \in [d]$, we define $p_\ell = e_\ell(R')/(2^{d - 1} m^2)$. Now \cref{prop:Rmd-richness} guarantees that, provided $m$ is sufficiently large in terms of $d$ and \nolinebreak $\eps$,
    \begin{equation*}
        \tfrac{1}{d} \sum_\ell p_\ell \geq \edged + 8 \eps.
    \end{equation*}
    Let $\eta$ be the parameter obtained from the graph building lemma, \cref{lemma:graph_building}, applied to $F$ with density $\epsilon 2^{-d}$. We may and will assume that $\eta \leq \epsilon 2^{-d}$. Note that the blocks $\cB \coloneqq \set{B_x \colon x \in \bit^d}$ form an equipartition of $R'$ into $2^d$ parts. 
    Let $L\in \N$ be obtained from the refined regularity lemma, \cref{lemma:regularity}, applied to $\eta$ and $2^d$ and let $\cP$ be the resulting $\eta$-regular equipartition of $R'$ that refines $\cB$. Note that $L$ is a function of $\eps$, $d$, and $F$ only. Provided $m$ is sufficiently large compared to $L$, the equipartition $\cP$ has split each block into an equal number of parts. We call this number $t$ (and so $\abs{\cP} = 2^d t$ and each part has size $\floor{m/t}$ or $\ceil{m/t}$) and note that 
    $t \leq L$.

    We clean $R'$ as follows:
    \begin{itemize}[noitemsep]
        \item we delete all edges between pairs of parts whose density is at most $\epsilon 2^{-d}$,
        \item we delete all edges between pairs of parts that are not $\eta$-regular.
    \end{itemize}
    The total number of deleted edges is at most
    \begin{align*}
        \epsilon 2^{-d} \abs{V(R')}^2 + \eta \tbinom{2^d t}{2} \ceil[\Big]{\tfrac{\abs{V(R')}}{2^d t}}^2 & \leq (\epsilon 2^{-d} + \eta) \abs{V(R')}^2 \\
        & \leq 2 \epsilon \cdot 2^{-d} (m 2^d)^2 = 4 \epsilon \cdot 2^{d - 1} m^2.
    \end{align*}
    Let $R''$ be the resulting graph. For each level $\ell$, we have
    \begin{equation*}
        \frac{e_\ell(R'')}{2^{d - 1} m^2} \geq \frac{e_\ell(R')}{2^{d - 1} m^2} - 4 \epsilon = p_\ell - 4 \epsilon.
    \end{equation*}
    Consider the graph $H$ whose vertex set is $\cP$ and where two parts are adjacent if they have an edge between them in $R''$ (in fact, due to the cleaning process, the pair of parts will be $\eta$-regular with density at least $\epsilon 2^{-d}$). We will order $H$ as follows: parts in different blocks are ordered in the same way as their blocks (so if $P \subset B_x$ and $Q \subset B_y$ where $x < y$, then $P < Q$) and parts within blocks are ordered arbitrarily.
    Since each block is split into $t$ parts by $\cP$, we may view $H$ as having vertex set $\bit^d \times [t]$ and talk about its level-$\ell$ edges for every $\ell \in [d]$.

    Now, provided $m$ is sufficiently large compared to $L$, property (ii) of \cref{prop:Rmd-properties} implies that, between two parts $P, Q \in \cP$, the number of level-$\ell$ edges in 
    $R = R(m, d)$ is at most
    \begin{equation*}
        (1 + \epsilon) 2^{-d + \ell} \abs{P} \abs{Q} \leq (1 + \epsilon) 2^{-d + \ell} \ceil{m/t}^2 \leq (1 + 2 \epsilon) 2^{-d + \ell} (m/t)^2.
    \end{equation*}
    As $R'' \subseteq R$, the number of level-$\ell$ edges of $R''$ between any two parts is also at most $(1 + 2 \epsilon) 2^{-d + \ell} (m/t)^2$. This implies that, for every $\ell \in [d]$
    \begin{align*}
        e_\ell(H) & \geq \frac{e_\ell(R'')}{(1 + 2 \epsilon) 2^{-d + \ell} (m/t)^2} \\
        & = \frac{1}{1 + 2 \epsilon} \cdot \frac{e_\ell(R'')}{2^{d - 1} m^2} \cdot \pairs_{\ell, d} \cdot t^2 \\
        & \geq \frac{1}{1 + 2 \epsilon} \cdot (p_\ell - 4 \epsilon) \cdot \pairs_{\ell, d} \cdot t^2 \\
        & \geq (p_\ell - 6 \epsilon) \cdot \pairs_{\ell, d} \cdot t^2.
    \end{align*}
    For each $x \in \bit^d$, we independently pick a random vertex $v_x$ of $H$ from $\set{x} \times [t]$. Let $G = H[\set{v_x \colon x \in \bit^d}]$ and note that $G$ can be viewed as an ordered graph on $\bit^d$. Furthermore, each edge of $H$ is in $G$ with probability $1/t^2$ and so
    \begin{equation*}
        \bE\biggl[\tfrac{1}{d} \sum_\ell \frac{e_{\ell}(G)}{\pairs_{\ell, d}}\biggr] = \tfrac{1}{d} \sum_{\ell} \frac{e_{\ell}(H)}{\pairs_{\ell, d}  \cdot t^2} \geq \tfrac{1}{d} \sum_{\ell} (p_\ell - 6 \eps) \geq \edged + 2\eps.
    \end{equation*}
    Thus there is a choice of the $v_x$ for which $G$ satisfies $\tfrac{1}{d} \sum_\ell e_{\ell}(G)/\pairs_{\ell, d} \geq \edged + 2\eps$. Then, since $e_\ell(G)/\pairs_{\ell, d} \in [0, 1]$, at least an $\eps$-proportion of 
    $\ell \in [d]$ satisfy $e_\ell(G)/\pairs_{\ell, d} \geq \edged + \varepsilon$. As
    $d \geq C/\varepsilon$, it follows that 
    $G$ is $(\edged + \eps, C)$-rich, and hence contains a copy of $F$ by the choice of $C$. 
    This means that $H$ contains $F$ as an ordered subgraph with vertices in distinct blocks. Recall that the edges of $H$ correspond to pairs of parts in $\cP$ that are $\eta$-regular and have density at least $\eps 2^{-d}$ in $R'$. By our choice of $\eta$, \cref{lemma:graph_building} guarantees that $R'$ contains a copy of $F$. Due to the way we ordered $H$ and the fact that the vertices of $F$ are all in distinct blocks, this copy of $F$ in $R'$ is ordered in the correct way. Thus $R'$ contains $F$ as an ordered subgraph, as required.
\end{proof}

\section{Uniformity versus richness}\label{sec:uni-rich}

This section is dedicated to showing that $\puni(F) \geq \pi_{\rm rich}(F)$. We will need a notion of edges being uniformly distributed and so develop some theory of homogeneous functions in \cref{subsec:homogeneity}. We use this to prove $\puni(F) \geq \pi_{\rm rich}(F)$ in \cref{subsec:uni>rich}.

\subsection{Homogeneity}\label{subsec:homogeneity}

Let $f \colon (\bit^d)^2 \to [0,1]$, $\eps>0$, and $0 \leq a \leq b \leq d$ integers. For two intervals $I_0, I_1\in \cP_a$ we say that $f$ is \defn{$(\eps, I_0, I_1, b)$-homogeneous} if 
\begin{equation*}
    \bE_{J_0\in \cP_b(I_0), J_1\in \cP_b(I_1)}\abs{\bE_{x\in J_0, y\in J_1} f(x,y) - \bE_{x\in I_0, y\in I_1} f(x,y)} \leq \eps.
\end{equation*}
Furthermore, we say that $f$ is \defn{$(\eps,a,b)$-homogeneous} if, for all but an $\eps$-proportion of $(I_0, I_1) \in \cP_a^2$, $f$ is $(\eps, I_0, I_1, b)$-homogeneous.

This first lemma shows that homogeneity is preserved when reducing $b$ or increasing $a$.

\begin{lemma}[inheritance] \label{lemma:homogeneity-inheritance}
    Let $\eps > 0$, $a \leq \ell \leq b$, and $I_0, I_1 \in \cP_a$. If $f$ is $(\eps^2/2, I_0, I_1, b)$-homogeneous, then
    \begin{itemize}
        \item $f$ is $(\eps^2/2, I_0, I_1, \ell)$-homogeneous\textup{;}
        \item for all but an $\eps$-proportion of $(J_0, J_1) \in \cP_\ell(I_0) \allowbreak \times \cP_\ell(I_1)$, $f$ is $(\eps, J_0, J_1, b)$-homogeneous and
        \begin{equation*}
            \abs{\bE_{J_0 \times J_1} f - \bE_{I_0 \times I_1} f} \leq \eps.
        \end{equation*}
    \end{itemize}
\end{lemma}

\begin{proof}
    Let $f$ be $(\eps^2/2, I_0, I_1, b)$-homogeneous. Fix $I_0, I_1 \in \cP_a$. 
    Define, for $J_0 \in \cP_\ell(I_0)$ and $J_1 \in \cP_\ell(I_1)$,
    \[
    E_{J_0, J_1} \coloneqq \bE_{K_0 \in \cP_b(J_0), K_1 \in \cP_b(J_1)} \abs{\bE_{K_0 \times K_1} f - \bE_{I_0 \times I_1} f}.
    \]
    Since a uniformly random element of $\cP_b(I_0)$ can be sampled by first choosing a uniformly random $J_0 \in \cP_\ell(I_0)$ and then choosing a uniformly random element of $\cP_b(J_0)$,
    \begin{equation}\label{eq:E}
        \bE_{J_0 \in \cP_\ell(I_0), J_1 \in \cP_\ell(I_1)} E_{J_0, J_1} = \bE_{K_0 \in \cP_b(I_0), K_1 \in \cP_b(I_1)} \abs{\bE_{K_0 \times K_1} f - \bE_{I_0 \times I_1} f} \leq \eps^2/2.
    \end{equation}
    Also, for each $J_0 \in \cP_\ell(I_0)$, a uniformly random element of $J_0$ can be sampled by first choosing a uniformly random $K_0 \in \cP_b(J_0)$ and then choosing a uniformly random element of $K_0$, so
    \begin{equation} \label{eq:ff}
        \abs{\bE_{J_0 \times J_1} f - \bE_{I_0 \times I_1} f} = \abs{\bE_{K_0 \in \cP_b(J_0), K_1 \in \cP_b(J_1)} [\bE_{K_0 \times K_1} f - \bE_{I_0 \times I_1} f]} \leq E_{J_0, J_1}.
    \end{equation}
    Thus,
    \[
        \bE_{J_0 \in \cP_\ell(I_0), J_1 \in \cP_\ell(I_1)} \abs{\bE_{J_0 \times J_1} f - \bE_{I_0 \times I_1} f}
        \leq \bE_{J_0 \in \cP_\ell(I_0), J_1 \in \cP_\ell(I_1)} E_{J_0, J_1} \leq \eps^2/2,
    \]
    which gives the first bullet point. Next, inequality~\eqref{eq:E} and Markov's inequality imply that all but an $\eps$-proportion of $(J_0, J_1) \in \cP_\ell(I_0) \times \cP_\ell(I_1)$ satisfy $E_{J_0, J_1} \leq \eps/2$. Suppose that $E_{J_0, J_1} \leq \eps/2$. Then \eqref{eq:ff} certainly implies that $\abs{\bE_{J_0 \times J_1} f - \bE_{I_0 \times I_1} f} \leq \eps/2 \leq \eps$. Also, by the triangle inequality,
    \begin{align*}
        & \bE_{K_0 \in \cP_b(J_0), K_1 \in \cP_b(J_1)} \abs{\bE_{K_0 \times K_1} f - \bE_{J_0 \times J_1} f} \\
        \leq \ & \bE_{K_0 \in \cP_b(J_0), K_1 \in \cP_b(J_1)} \bigl(\abs{\bE_{K_0 \times K_1} f - \bE_{I_0 \times I_1} f} + \abs{\bE_{I_0 \times I_1} f - \bE_{J_0 \times J_1} f}\bigr) \\
        \overset{\eqref{eq:ff}}{\leq} \ & 2E_{J_0, J_1} \leq \eps,
    \end{align*}
    and so $f$ is $(\eps, J_0, J_1, b)$-homogeneous, as required.
\end{proof}

The next lemma shows that, for any finite collection of bounded functions on $(\bit^d)^2$, if $d$ is large enough, then there are far apart $a$ and $b$ such that the functions are all $(\eps, a, b)$-homogeneous.

\begin{lemma} \label{lemma:bihomogeneity}
    For all $\eps>0$ and $s,C\in \bN$, there exists $C'\in \N$ such that the following holds.
    Let $f_1,\dots,f_s \colon (\bit^d)^2 \to [0, 1]$ be functions, and let $\cL \subset [d]$ have size at least $C'$. Then there exist $1 \leq a < b \leq \max \cL$ such that $\abs{(a,b] \cap \cL} \geq C$ and $f_1, \dots, f_s$ are all $(\eps, a, b)$-homogeneous.
\end{lemma}

\begin{proof}
    Let $t = \ceil{s \eps^{-4}}$ and  $C' = C t + 1$. Let 
    $X = X_1 \dots X_d$ and 
    $Y = Y_1 \dots Y_d$ be uniformly-random independent elements of $\bit^d$ (so 
    $X_i,Y_i \in \bit$ for $i \in [d]$). For each $j\in[s]$, define the Doob martingale
    \begin{align*}
        Z_i^{(j)} & \coloneqq \bE[f_j(X, Y) \mid X_1, \dots, X_i, Y_1, \dots, Y_i] \qquad \text{for } i = 1, \dots, d, \\
        Z_0^{(j)} & \coloneqq \bE[f_j(X, Y)].
    \end{align*}
    For a fixed value of $j$, each $Z_i^{(j)} \colon \bit^{2i} \to [0,1]$ is a random variable which is a function of $X_1, \dots, X_i, Y_1, \dots, Y_i$. If $P, Q \in \cP_i$ are fundamental intervals at level $i$, then every $x \in P$ has the same coordinates at positions $1, \dots, i$ and similarly for every $y \in Q$. Hence, all $(x, y) \in P \times Q$ are mapped to the same image under $Z_i^{(j)}$ and this value is
    \begin{equation}\label{eq:eval-Zij}
        Z_i^{(j)}(x, y) = \bE_{P \times Q} f_j
    \end{equation}
    by the definition of conditional expectation. In particular, as the intervals of $\cP_d$ are simply the individual elements of $\bit^d$, we have $Z_d^{(j)} = f_j$. Define the vector-valued martingale
    \begin{equation*}
        \mathbf{Z}_i = \bigl(Z_i^{(1)}, \dots, Z_i^{(s)} \bigr) \in [0, 1]^s
    \end{equation*}
    for all $0 \le i \le d$. Since $\cL$ has size at least 
    $C' = C t + 1$ there is a sequence 
    $1 = a_0 < a_1 < \dots < a_t = \max \cL$ such that  $\abs{(a_i, a_{i + 1}] \cap \cL} \geq C$ for each $i$. Since each $f_j$ is $[0,1]$-valued, the random variable $f_j(X, Y)$ has variance at most $1$, and hence
    \begin{align*}
        s & \ge \sum_{j=1}^s \var f_j(X, Y) = \sum_{j=1}^s \bE\Bigl[\bigl(Z_d^{(j)} - Z_0^{(j)}\bigr)^2\Bigr] \\
        & = \bE\bigl[\norm{\mathbf{Z}_d - \mathbf{Z}_0}^2 \bigr] \\
        & = \bE \Biggl[\norm[\bigg]{(\mathbf{Z}_1 - \mathbf{Z}_0) + (\mathbf{Z}_d - \mathbf{Z}_{a_t}) + \sum_{i = 0}^{t - 1} (\mathbf{Z}_{a_{i + 1}} - \mathbf{Z}_{a_i})}^2 \Biggr]\\
        & = \bE \Biggl[\norm{\mathbf{Z}_1 - \mathbf{Z}_0}^2 + \norm{\mathbf{Z}_d - \mathbf{Z}_{a_t}}^2 + \sum_{i=0}^{t - 1} \norm{\mathbf{Z}_{a_i} - \mathbf{Z}_{a_{i + 1}}}^2 \Biggr] \\
        & \geq \sum_{i = 0}^{t - 1}\bE\bigl[ \norm{\mathbf{Z}_{a_i} - \mathbf{Z}_{a_{i+1}}}^2\bigr],
    \end{align*}
    where the last equality uses the Pythagorean law of orthogonal martingale differences. Since all the summands are non-negative and there are 
    $t \ge s \eps^{-4}$ of them, there exists some $i$ such that
    \begin{equation*}
        \bE\Bigl[\norm{\mathbf{Z}_{a_i} -\mathbf{Z}_{a_{i+1}}}^2\Bigr] \le \frac{s}{t} \le \eps^4.
    \end{equation*}
    Setting $a = a_i$ and $b = a_{i + 1}$, we have $\abs{(a, b] \cap \cL} \geq C$ and, for each $j \in [s]$,
    \begin{equation*}
       \bE\Bigl[\bigl(Z_a^{(j)} - Z_b^{(j)}\bigr)^2\Bigr] \le \bE\Bigl[\norm{\mathbf{Z}_{a} - \mathbf{Z}_{b}}^2\Bigr] \le \eps^4.
    \end{equation*}
    We will argue that $f_j$ is $(\eps,a,b)$-homogeneous.

    Let $I_a, I'_a \in \cP_a$ be sampled uniformly at random and then select $I_b \in \cP_b(I_a), I'_b \in \cP_b(I'_a)$ uniformly at random as well. Observe that this is identical to sampling $I_b, I'_b \in \cP_b$ uniformly at random. Hence, by~\eqref{eq:eval-Zij} and convexity of $t \mapsto t^2$, we have
    \begin{align*}
        \eps^4 \ge \bE\Bigl[\bigl(Z_a^{(j)} - Z_b^{(j)}\bigr)^2\Bigr] &=
        \bE_{I_a, I'_a \in \cP_a}  \bE_{I_b \in \cP_b(I_a)} \bE_{I'_b \in \cP_b(I'_a)} \bigl(\bE_{I_a \times I'_a} f_j - \bE_{I_b \times I'_b} f_j\bigr)^2 \\
        & \geq \Bigl(\bE_{I_a, I'_a \in \cP_a} \bE_{I_b \in \cP_b(I_a)} \bE_{I'_b \in \cP_b(I'_a)} \abs{\bE_{I_a \times I'_a} f_j - \bE_{I_b \times I'_b} f_j}\Bigr)^2,
    \end{align*}
    and so
    \begin{equation*}
        \bE_{I_a, I'_a \in \cP_a} \bE_{I_b \in \cP_b(I_a)} \bE_{I'_b \in \cP_b(I'_a)} \abs{\bE_{I_a \times I'_a} f_j - \bE_{I_b \times I'_b} f_j} \leq \eps^2.
    \end{equation*}
    Markov's inequality then implies that all but an $\eps$-proportion of $(I_a, I'_a) \in \cP_a^2$ satisfy
    \begin{equation*}
        \bE_{I_b \in \cP_b(I_a)} \bE_{I'_b \in \cP_b(I'_a)} \abs{\bE_{I_a \times I'_a} f_j - \bE_{I_b \times I'_b} f_j} \leq \eps.
    \end{equation*}
    That is, $f_j$ is $(\eps, a, b)$-homogeneous. But $j \in [s]$ was arbitrary, as required.
\end{proof}

The final lemma in this subsection is a simple density increment lemma.

\begin{lemma}[density increment] \label{lemma:interval-density-increase}
    For all $\edged, \gamma, C > 0$, the following holds. Let $G$ be a graph on $\bit^d$, let $0 \leq a < d$, and let $\cL \subset(a,d]$ be a set of $\edged$-rich levels for $G$ with $\abs{\cL} \geq C$. Suppose that $G$ does not contain an $(\edged+\gamma^3/4, \gamma^3 C/4)$-rich graph. Then, for all but a $\gamma$-proportion of intervals $I\in \cP_a$, all but a $\gamma$-proportion of $\ell \in \cL$ satisfy
    \begin{equation*}
        e_{\ell,G}(I)\geq 
        (\edged-\gamma)\pairs_{\ell}(I),
    \end{equation*}
    that is, level $\ell - a$ is $(\alpha - \gamma)$-rich with respect to $G[I]$.\footnote{Note that level $\ell-a$ in $G[I]$ corresponds to level $\ell$ in $G$, as 
    $I \in \cP_a$.}
\end{lemma}

\begin{proof}
    We may assume $\gamma \in (0, 1)$ since the conclusion is immediate for $\gamma \geq 1$. Define, for $\ell \in \cL$ and $I \in \cP_a$, $f(\ell, I) \coloneqq \frac{e_{\ell, G}(I)}{\pairs_\ell(I)} \in [0, 1]$.
    First note that each $\ell \in \cL$ is $\edged$-rich and so satisfies
    \begin{equation}\label{eq:expf}
        \edged \leq \frac{e_\ell(G)}{\pairs_{\ell, d}} = \sum_{I \in \cP_a} \frac{e_{\ell, G}(I)}{\pairs_{\ell, d}} = \bE_{I \in \cP_a} \frac{e_{\ell, G}(I)}{\pairs_{\ell, d}/\abs{\cP_a}} = \bE_{I \in \cP_a} f(\ell, I).
    \end{equation}
    If there is some $I \in \cP_a$ for which at least a $\gamma^3/4$-proportion of $\ell \in \cL$ satisfy $f(\ell, I) \geq \edged + \nolinebreak \gamma^3/4$, then $G[I]$ is $(\edged + \gamma^3/4, \gamma^3C/4)$-rich, contradicting the assumption that $G$ contains no such subgraph. Thus, every $I \in \cP_a$ satisfies
    \begin{equation*}
        \bE_{\ell \in \cL} f(\ell, I) \leq 1 \cdot (\edged + \gamma^3/4) + \gamma^3/4 \cdot 1  = 
        \edged + \gamma^3/2.
    \end{equation*}
    If at least a $\gamma$-proportion of $I \in \cP_a$ satisfy $\bE_{\ell \in \cL} f(\ell, I) < \edged - \gamma^2/2$, then
    \begin{equation*}
        \bE_{I \in \cP_a} \bE_{\ell \in \cL} f(\ell, I) \leq \gamma (\edged - \gamma^2/2) + (1 - \gamma)(\edged + \gamma^3/2) < \edged,
    \end{equation*}
    which contradicts \eqref{eq:expf}. Hence, all but a $\gamma$-proportion of $I \in \cP_a$ satisfy 
    $\bE_{\ell \in \cL} f(\ell, I) \geq \edged - \gamma^2/2$. Fix such an $I \in \cP_a$. It suffices to show that all but a $\gamma$-proportion of $\ell \in \cL$ have $f(\ell, I) \geq \alpha - \gamma$. If this does not occur, then
    \begin{equation*}
        \bE_{\ell \in \cL} f(\ell, I) \leq \gamma^3/4 \cdot 1 + \gamma \cdot (\edged - \gamma) + (1 - \gamma)(\edged + \gamma^3/4) < \edged + \gamma^3/2 - \gamma^2 \leq \edged - \gamma^2/2,
    \end{equation*}  
    where we used the fact that at most a $\gamma^3/4$-proportion of $\ell$ have $f(\ell, I) \geq \edged + \gamma^3/4$. This is the required contradiction.
\end{proof}

\subsection{\texorpdfstring{$\puni \geq \prich$}{Uniform Turán density at least rich Turán density}}\label{subsec:uni>rich}

We are now ready to prove that $\puni \geq \prich$. We will do this using objects called immersions (defined below). First, a little notation: for a binary string $x$, \defn{$\abs{x}$} denotes the length of $x$ and \defn{$\prI{x}$} denotes the fundamental interval in $\cP_{\abs{x}}$ of strings in $\bit^d$ that start with prefix $x$. For two sets of binary strings $A, B$, we write $\delta(A,B) = \ell$ to mean that $\delta(a,b)=\ell$ for all $a\in A$ and $b\in B$.

Let $G$ be a graph on $\bit^d$. A \defn{$(k, \alpha, \eps, C)$-immersion in $G$} is a tuple $(w,h,\cL,\iota)$ where 
$k \leq w < h \leq d$, $\cL\subset (w,h]$ has size at least $C$, 
and $\iota\colon \bit^k \to \bit^w$ is an increasing map (with respect to the lexicographic orders) having the following properties.

\begin{enumerate}[label=\upshape{(P\arabic*)}, leftmargin =\widthof{P10000}]
    \item For all distinct $u,v\in \bit^k$, the number of edges of $G$ between $\prI{\iota(u)}$ and $\prI{\iota(v)}$ is at least $\alpha \abs{\prI{\iota(u)}}\abs{\prI{\iota(v)}}$. \label{axiom:top-richness}
    \item For all distinct $u,v\in \bit^k$, the indicator function of the edges of $G$ is \linebreak
    $(\eps, \prI{\iota(u)}, \prI{\iota(v)},h)$-homogeneous.\label{axiom:homogeneity}
    \item For all $v\in \bit^k$, the induced graph $G[\prI{\iota(v)}]$ is $\alpha$-rich on level $\ell-w$ for every $\ell\in \cL$.\label{axiom:bottom-richness}\footnote{Note that level $\ell-w$ in $G[I_{\iota(v)}]$ corresponds to level $\ell$ in $G$, as 
    $I_{\iota(v)} \in \cP_w$.}
\end{enumerate}

Intuitively, a $(k, \alpha, \eps, C)$-immersion is a collection of $2^k$ fundamental intervals (all at the same level $w$) such that the edges between any two intervals are pseudorandom (\ref{axiom:homogeneity}) and have density at least $\edged$ (\ref{axiom:top-richness}). Furthermore, within each interval there are lots of $\edged$-rich levels (\ref{axiom:bottom-richness}).
The following is the main lemma of this section, and is used to build up larger and larger immersions (increasing $k$) at a small cost to the density.

\begin{lemma}\label{lemma:immersion-step}
    For all $\eps>0$ and $k \in \bZ_{\geq 0}$, there exists $\eta>0$ such that for all 
    $C,D \in \N$ there exists $C'\in \bN$ satisfying the following.
    
    Let $d \in \N$ and $\alpha>\eps$, and suppose that a graph $G$ on $\bit^d$ has a $(k,\alpha,\eta, C')$-immersion but contains no 
    $(\alpha + \eta, D)$-rich subgraph. Then $G$ has a 
    $(k+1,\alpha - \eps,\eps, C)$-immersion as well.
\end{lemma}

\begin{proof}
    Let $\eps > 0$ and $k \in \bZ_{\geq 0}$. Fix $\eta \ll \eta_2 \ll \eta_3 \ll \eta_4 \ll \eps,1/(k + 1)$ 
    and $C' \gg C'' \gg C,D,1/\eta$.
    Let $\alpha > \eps$ and $(w,h,\cL,\iota)$ be a $(k,\alpha,\eta, C')$-immersion in $G$, where $G$ is a graph on 
    $\bit^d$. Let $m=\min \cL$.

    \begin{claim} \label{claim:immersion1}
       There exists a family $(I_v)_{v\in \bit^k}$ of intervals, where $I_v\in \cP_{m-1}(\prI{\iota(v)})$, such that the following hold.

        \begin{itemize}
            \item For all distinct $u,v\in \bit^k$, the function $\boldsymbol{1}(\bullet \in E(G))$ is $(\eta_2, I_u, I_v, h)$-homogeneous and has density at least $\alpha - \eta_2$ on $I_u \times I_v$.
            \item For all $v\in \bit^k$, the number of edges at level $m$ in $I_v$ is at least $(\alpha-\eta_3)\tau_{m}(I_v)$.
            \item There exists a set $\cL_1\subset (m,h]\cap \cL$ of size at least $C'/2$ such that for all $v\in \bit^k$ and $\ell\in \cL_1$, the level $\ell-m+1$ is $(\alpha - \eta_2)$-rich in $G[I_v]$.
        \end{itemize}
    \end{claim}

    \begin{proofclaim}
        For each $v\in \bit^k$, we choose the set $I_v\in \cP_{m-1}(I_{\iota(v)})$ uniformly at random independently from the others. We show that, with positive probability, the family $(I_v)_{v\in \bit^k}$ has all the desired properties.

        For $v \in \bit^k$, let $A_v$ be the event that $I_v$ has at least $(\alpha - \eta_3) \pairs_m(I_v)$ edges at level $m$. Since $m \in \cL$, the number of level-$m$ edges in $\prI{\iota(v)}$ overall is at least $\alpha \pairs_m(\prI{\iota(v)})$ for each $v$ by \ref{axiom:bottom-richness}. Note that every level-$m$ edge in $I_{\iota(v)}$ is contained in exactly one $I \in \cP_{m - 1}(I_{\iota(v)})$ and so $\bE_{I_v} [e_{m, G}(I_v) \pairs_m(I_v)^{-1}] \geq \alpha$. Since $e_{m, G}(I_v) \pairs_m(I_v)^{-1} \in [0, 1]$, an averaging argument shows $\bP(A_v) = \bP(e_{m, G}(I_v) \tau_m(I_v)^{-1} \geq \alpha - \eta_3) \geq \eta_3$.
        Since the events $(A_v)_{v\in \bit^k}$ are independent, $\bP(\bigcap_{v\in \bit^k} A_v) \geq \eta_3^{2^k}$.

        Now, let $\cL_v\subset \cL$ be the set of $\ell \in \cL$, for which the level $\ell - m + 1$ is $(\alpha - \eta_2)$-rich in $G[I_v]$. Further, let $B_v$ be the event that $\cL_v$ has size at least 
        $(1-\eta_2)\abs{\cL}$. By \ref{axiom:bottom-richness}, the induced subgraph $G[\prI{\iota(v)}]$ is $\alpha$-rich on level $\ell - w$ for every $\ell \in \cL$. 
        Also, by assumption, $G$ has no $(\alpha + \eta, \eta C')$-rich subgraph, as $\eta C' \geq D$ by 
        $C' \gg D,1/\eta$. Thus, \cref{lemma:interval-density-increase} (applied with $\gamma \coloneqq (4\eta)^{1/3}$) 
        implies that 
        $\abs{\cL_v} \geq (1-\gamma)\abs{\cL} \geq (1-\eta_2)\abs{\cL}$ with probability at least $1-\gamma \geq 1-\eta_2$, and hence
        $B_v$ has probability at least $1-\eta_2$. (Here we used that 
        $\gamma \leq \eta_2$, as $\eta \ll \eta_2$.)
    
        Finally, for distinct $u, v \in \bit^k$, let $C_{u,v}$ be the event that $\boldsymbol{1}(\bullet \in E(G))$ is $(\eta_2, I_u, I_v, h)$-homogeneous and has density at least $\alpha - \eta_2$ on $I_u \times I_v$.
        By \ref{axiom:top-richness} and \ref{axiom:homogeneity}, the function $\boldsymbol{1}(\bullet \in E(G))$ is $(\eta, \prI{\iota(u)}, \prI{\iota(v)}, h)$-homogeneous with density at least $\alpha$ on $\prI{\iota(u)} \times \prI{\iota(v)}$ and so the second bullet point of \cref{lemma:homogeneity-inheritance} implies that $\bP(C_{u,v}) \geq 1 - \eta_2$.

        Employing a union bound, we see that
        \begin{equation*}
            \bP\Biggl(\bigcap_{v\in \bit^k} A_v\cap \bigcap_{v\in \bit^k} B_v \cap \bigcap_{uv\in (\bit^k)^{(2)}} C_{u,v}\Biggr)\geq \eta_3^{2^k} - 2^k\eta_2-2^{2k}\eta_2 >0,
        \end{equation*}
        which means there exists a family $(I_v)_{v\in \bit^k}$ such that $A_v$, $B_v$, and $C_{u,v}$ are satisfied for all $u,v\in \bit^k$. It is obvious that the family satisfies the first and second bullet points in the claim, and to see that they satisfy the third, let $\cL_1=\bigcap_{v\in \bit^k}\cL_v \setminus \set{m}$ and note that 
        $\abs{\cL_1}\geq 
        (1-2^k\eta_2) \abs{\cL} - 1 \geq C'/2$, as $\abs{\cL} \geq C'$.
    \end{proofclaim}
    
    \begin{claim}\label{claim:immersion2}
        There exist an interval $(a,b]\subset [m,h]$, a set $\cL_2\subset (a,b]\cap \cL_1$ of size at least $C$, and for each $v\in \bit^k$ two intervals $J_v, K_v\in \cP_{a}(I_v)$ with $\delta(J_v,K_v) = m$ such that the following hold.

        \begin{itemize}
            \item For all distinct $u,v\in \bit^k$ and all sets $X\in \set{J_u,K_u}$ and $Y\in \set{J_v,K_v}$, the function $\boldsymbol{1}(\bullet \in E(G))$ is $(\eps, X, Y, b)$-homogeneous and has density at least $\alpha-\eps$ on $X \times Y$.
            \item For all $v \in \bit^k$, the function 
            $\boldsymbol{1}(\bullet \in E(G))$ is 
            $(\eps, J_v, K_v,b)$-homogeneous and has density at least $\alpha-\eps$ on $J_v \times K_v$.
            \item For all $v\in \bit^k$ and $\ell\in \cL_2$, the level 
            $\ell-a$ is $(\alpha - \eps)$-rich in $G[J_v]$ and in $G[K_v]$.
        \end{itemize}
    \end{claim}

    \begin{proofclaim}
        Since each $I_v \in \cP_{m - 1}$, the level-$m$ edges on each $I_v$ are between $\lhs(I_v)$ and $\rhs(I_v)$. We will choose $J_v \subseteq \lhs(I_v)$ and $K_v \subseteq \rhs(I_v)$ which ensures $\delta(J_v, K_v) = m$. 
        
        For $v \in \bit^k$, let $f_v \colon I_v^2 \to [0, 1]$ be the indicator function for the level-$m$ edges (with respect to $G$) on $I_v$, where we identify $I_v$ with $\bit^{d - m + 1}$. Since 
        $\cL_1 \subset (m, h]$ we have $\cL_1 - (m - 1) \subset (1, h - m + 1]$. Applying \cref{lemma:bihomogeneity} 
        (with $s=2^k$) to $\cL_1 - (m - 1)$ and the $f_v$ ($v \in \bit^k$) shows that, provided $C'$ is sufficiently large, there exist integers $a, b$ with $m \leq a < b \leq h$ such that $\abs{(a, b] \cap \cL_1} \geq C''$ and the $f_v$ are all $(\eta, a - m + 1, b - m + 1)$-homogeneous.
        Put $\cL_1' \coloneqq (a,b] \cap \cL_1$. 

        We now choose intervals $J_v \in \cP_a(\lhs(I_v))$ and $K_v \in \cP_a(\rhs(I_v))$ uniformly at random for each $v\in \bit^k$, independently. Using very similar arguments as for the previous claim, we show that with positive probability the chosen subintervals have all the desired properties. 

        First, for each $v \in \bit^k$, the function 
        $f_v$ is 
        $(\eta, a - m + 1, b - m + 1)$-homogeneous.
        By the definition of homogeneity, this implies that, for a uniformly random pair of intervals $I,I' \in \cP_a(I_v)$, the probability that $f_v$ is 
        $(\eta, I, I', b)$-homogeneous is at least $1-\eta$. Note that $J_v,K_v$ are such random intervals, with the additional restriction that 
        $J_v \subseteq \lhs(I_v)$ and $K_v \subseteq \rhs(I_v)$. 
        This restriction has probability $1/4$ so conditioning on it increases the failure probability by at most a factor of four.
        Hence, with probability at least 
        $1-4\eta$, the function
        $f_v$
        is $(\eta, J_v, K_v, b)$-homogeneous.
        Since $f_v$ agrees with $\boldsymbol{1}(\bullet \in E(G))$ on $J_v \times K_v$,
        the function $\boldsymbol{1}(\bullet \in E(G))$ is also $(\eta, J_v, K_v, b)$-homogeneous and so
        $(\eps, J_v, K_v, b)$-homogeneous.

        Next, for distinct $u, v \in \bit^k$, the function $\boldsymbol{1}(\bullet \in E(G))$ is $(\eta_2, I_u, I_v, h)$-homogeneous and has density at least $\alpha - \eta_2$ on $I_u \times I_v$ by \cref{claim:immersion1}.
        By the second bullet point of \cref{lemma:homogeneity-inheritance}, for a uniformly random pair $(X, Y) \in \cP_a(I_u) \times \cP_a(I_v)$, with probability at least $1 - \eta_3/16$, the function $\boldsymbol{1}(\bullet \in E(G))$ is $(\eta_3/16, X, Y, h)$-homogeneous and has density at least $\edged - \eta_2 - \eta_3/16 \geq \edged - \eta_3$ on $X \times Y$.
        Conditioning $X, Y$ to lie in the prescribed halves and taking a union bound over the four choices $X \in \set{J_u, K_u}$ and $Y \in \set{J_v, K_v}$ shows that, with probability at least $1 - \eta_3$, 
        for all $X \in \set{J_u, K_u}$ and $Y \in \set{J_v, K_v}$,
        the function $\boldsymbol{1}(\bullet \in E(G))$ has density at least $\edged - \eta_3$ on $X \times Y$ and is $(\eta_3, X, Y, h)$-homogeneous and hence $(\eta_3, X, Y, b)$-homogeneous (by the first bullet point of \cref{lemma:homogeneity-inheritance}).

        By the second item of \cref{claim:immersion1}, $e_{m, G}(I_v) \geq (\edged - \eta_3) \tau_m(I_v)$ for every $v \in \bit^k$. 
        This means that the expected density of $G[J_v,K_v]$ is at least 
        $\edged - \eta_3$. Since densities are at most 1, an averaging argument shows that the density of $G[J_v,K_v]$ is at least 
        $\edged - \eta_3 - \eta_4 \geq \edged-\eps$ with probability at least $\eta_4$.
        Hence, with probability at least $\eta_4^{2^k}$, the density of $G[J_v,K_v]$ is at least 
        $\edged-\eps$ for every $v\in \bit^k$.

        Lastly, let $\cL_v^J, \cL_v^K \subseteq \cL_1'$ be the sets of 
        $\ell \in \cL_1'$ for which the level 
        $\ell - a$ is $(\alpha - \eps)$-rich in $G[J_v]$ and $G[K_v]$, respectively. By \cref{claim:immersion1}, $G[I_v]$ is $(\alpha - \eta_2)$-rich on level $\ell - m + 1$ for all $\ell \in \cL_1'$. Since $\eta_2 \gg \eta$, there exists no $((\alpha - \eta_2) + 2 \eta_2, (2 \eta_2) \cdot C'')$-rich subgraph of $G$, and hence of $G[I_v]$. Here we used that 
        $C'' \gg D,1/\eta$.
        Since $\abs{\cL_1'} \geq C''$, we may apply
        \cref{lemma:interval-density-increase} to $G[I_v]$ (with 
        $\gamma = (8\eta_2)^{1/3}$), to conclude that
        $\abs{\cL_v^J} \geq (1-\gamma)\abs{\cL_1'} \geq (1-\eta_3)\abs{\cL_1'}$ with probability at least $1-2\gamma$ (as $J_v$ is chosen uniformly from half of the intervals in $\cP_a(I_v)$, namely from $\lhs(I_v)$). The same holds for $\cL_v^K$.
        Hence, with probability at least 
        $1-4\gamma \geq 1 - \eta_3$, both sets $\cL_v^J, \cL_v^K$ have size at least $(1 - \eta_3) \abs{\cL_1'}$.

        Applying inclusion-exclusion and union bounds as in the previous claim (and using 
        $\eta \leq \eta_3 \ll \eta_4$) yields collections 
        $(J_v)_{v\in \bit^k}$ and $(K_v)_{v\in \bit^k}$ with the desired properties for $\cL_2=\bigcap_{v\in \bit^k} (\cL^J_v\cap \cL^K_v)$. Finally note that 
        $\abs{\cL_2} \geq 
        (1 - 2^{k + 1}\eta_3)\abs{\cL_1'} \geq C$, since 
        $\abs{\cL_1'} \geq C'' \gg C,1/\eta$.
    \end{proofclaim}

    Now define $\iota' \colon \bit^{k+1} \to \bit^a$ so that, writing $v = u v_{k + 1}$ for $u \in \bit^k$, 
    \begin{equation*}
        v \mapsto 
        \begin{cases}
            \pref_a(J_u) & \text{if } v_{k + 1} = 0, \\
            \pref_a(K_u) & \text{if } v_{k + 1} = 1.
        \end{cases}
    \end{equation*}
    Here, for an interval $I \in \cP_a$, we denote by \defn{$\pref_a(I)$} the unique string $x \in \bit^a$ such that $I = I_x$.
    Then $(a, b, \cL_2, \iota')$ is a $(k + 1, \alpha-\eps, \eps, C)$-immersion, as desired.
\end{proof}

It might seem unnecessary to separate \cref{claim:immersion1,claim:immersion2} as their proofs are very similar. Why can we not just sample the $J_v$ and $K_v$ without finding $I_v$ first? The key point is that choosing the intervals $I_v$ first allows us to apply \cref{lemma:bihomogeneity} to exactly $2^k$ functions. Applying it simultaneously to all candidate intervals in $\cP_{m - 1}(I_{\iota(v)})$ would involve a number of functions that can depend on $d$.

By iterating \cref{lemma:immersion-step}, we obtain the following:

\begin{lemma}\label{lemma:immersion-complete}
    For all $\eps > 0$ and $k \in \bZ_{\geq 0}$, there exists $\eta > 0$ such that for all $C, D \in \bN$ there exists 
    $C' \in \bN$ satisfying the following. Let $\alpha > \varepsilon$ and suppose that $G$ is an $(\alpha, C')$-rich graph that does not contain an $(\alpha+\eta,D)$-rich subgraph. Then $G$ has a 
    $(k, \alpha - \eps, \eps, C)$-immersion.
\end{lemma}

\begin{proof}
    We prove the statement by induction on $k$ for all $\eps>0$ and $C, D \in \bN$. For the induction start $k=0$, it is easy to see that every $(\alpha, C)$-rich graph has a $(0, \alpha, \eps, C)$-immersion for every $\eps > 0$. Now assume that the induction hypothesis holds for $k$ and prove it for $k+1$. Apply \cref{lemma:immersion-step} with $k$, $\eps_{\ref{lemma:immersion-step}} = \eps/2$, 
    $C_{\ref{lemma:immersion-step}} = C$ 
    and 
    $D_{\ref{lemma:immersion-step}} = D$ 
    to obtain 
    $\eta_{\ref{lemma:immersion-step}} \coloneq 
    \eta_{\ref{lemma:immersion-step}}(\varepsilon/2,k)$
    and $C'' \coloneq C'_{\ref{lemma:immersion-step}}(\varepsilon/2,k,C,D)$. Put 
    $\gamma\coloneq 
    \min(\eps/2, \eta_{\ref{lemma:immersion-step}}/2)$.

    Next, apply the induction hypothesis 
    to obtain $\eta_k = \eta_{\ref{lemma:immersion-complete}}(\gamma,k)$ and $C_k = C'_{\ref{lemma:immersion-complete}}(\gamma,k,C'',D)$ so that for every $\alpha > \gamma$, every 
    $(\alpha, C_k)$-rich graph with no 
    $(\alpha+\eta_k,D)$-rich subgraph
    has a 
    $(k, \alpha - \gamma, \gamma, C'')$-immersion.
    We may assume that 
    $\eta_k \leq \gamma$.

    Now let $\alpha > \varepsilon$ and let $G$ be an $(\alpha,C_k)$-rich graph with no $(\alpha+\eta_k,D)$-rich subgraph. By the choice of $\eta_k$ and $C_k$, we know that $G$ contains a 
    $(k, \alpha - \gamma, \gamma, C'')$-immersion.  
    Now, by the choice of $C''$ and $\gamma$, we get from \cref{lemma:immersion-step} 
    that $G$ has a $(k+1, \alpha - \gamma - \eps/2, \eps/2, C)$-immersion, which by monotonicity of the relevant parameters also is a $(k + 1, \alpha - \eps, \eps, C)$-immersion, as required.
    Here we applied \cref{lemma:immersion-step} with density 
    $\alpha - \gamma$, and used that $G$ contains no 
    $(\alpha-\gamma +\eta_{\ref{lemma:immersion-step}},D)$-rich subgraph, as 
    $
    \eta_{\ref{lemma:immersion-step}} - \gamma \geq \gamma \geq \eta_k$.
\end{proof}

\begin{theorem}\label{cor:rich<uni}
    For every ordered graph $F$, $\puni(F)\geq \prich(F)$.
\end{theorem}
\begin{proof}
    Suppose towards a contradiction that $\puni(F)<\prich(F)$, and let
    \begin{equation*}
        \eps\coloneq(\prich(F) - \puni(F))/4,
    \end{equation*}
    and let $\delta$ be sufficiently small and $N$ sufficiently large that any $(\puni(F)+\eps,\delta)$-interval-dense graph on at least $N$ vertices contains a copy of $F$. 
    Let $k\in \N$ be large in terms of $\eps$, $\delta$, and $N$, and let 
    $\eta \coloneq \eta_{\ref{lemma:immersion-complete}}(\varepsilon,k)$ be given by Lemma \ref{lemma:immersion-complete}. 
    We may assume that 
    $\eta \leq \varepsilon$.
    By the definition of $\prich(F)$, there exists an integer $D$ such that every $(\prich(F)+\eta/2,D)$-rich graph contains a copy of $F$. Finally, let 
    $C' \coloneq C'_{\ref{lemma:immersion-complete}}(\varepsilon,k,1,D)$ be given by \cref{lemma:immersion-complete} (with $C=1$).

    By the definition of $\prich$, there exists a $(\prich(F)-\eta/2, C')$-rich graph $G$ that does not contain $F$. 
    By the choice of $D$, this means that $G$ contains no $(\prich(F)+\eta/2,D)$-rich subgraph.
    Hence, by
    \cref{lemma:immersion-complete} (applied for $\alpha \coloneqq \prich(F) - \eta/2$), $G$ contains a \linebreak $(k,\prich(F)-\eta/2-\eps, \eps, 1)$-immersion $(w,h,\cL,\iota)$. Note that by choice of $\eps$ and $\eta$, we have $\prich(F)-\eta/2-\eps\geq \puni(F)+2\eps$. Consider the subgraph $H$ that $G$ induces on
    \begin{equation*}
        V(H) = \bigcup_{u\in \bit^k} I_{\iota(u)}.
    \end{equation*}
    We claim that $H$ is also $(\puni(F)+\eps, \delta)$-interval-dense, where $V(H)$ inherits its order from $V(G)$. Let $J_0$ and $J_1$ be two disjoint intervals of length at least $\delta \abs{V(H)}$. For both $i\in \bit$, because $\iota$ is increasing, the set of $u\in \bit^k$ such that $I_{\iota(u)}\subset J_i$ forms an interval $K_i\subset \bit^k$. Each interval $I_{\iota(u)}$ has length $\abs{V(H)}/2^k$ and so
    \begin{equation*}
        \abs{K_i} - 1 \geq \frac{\abs{J_i}}{\abs{V(H)}/2^k}-3 \geq \frac{\abs{J_i}}{\abs{V(H)}/2^k} \biggl(1-\frac{3}{2^k\delta}\biggr).
    \end{equation*}
    By \ref{axiom:top-richness}, for any distinct $u\in K_0$ and $v\in K_1$, the number of edges between $I_{\iota(u)}$ and $I_{\iota(v)}$ is at least 
    \begin{equation*}
        (\puni(F) + 2\eps) \abs{I_{\iota(u)}} \cdot \abs{I_{\iota(v)}} = (\puni(F) + 2\eps) \frac{\abs{V(H)}^2}{2^{2k}},
    \end{equation*}
    which means that
    \begin{align*}
        e_G(J_0, J_1) & \geq \abs{\set{(u, v) \in K_0 \times K_1 \colon u \neq v}} \cdot (\puni(F) + 2\eps) \frac{\abs{V(H)}^2}{2^{2k}} \\
        & \geq \abs{K_0} (\abs{K_1} - 1) \cdot (\puni(F) + 2\eps) \frac{\abs{V(H)}^2}{2^{2k}} \\
        & \geq (\puni(F) + 2 \eps) \biggl(1 - \frac{3}{2^k \delta}\biggr)^2 \abs{J_0} \abs{J_1}.
    \end{align*}
    Since $k$ was chosen large enough in terms of $\eps$ and $\delta$, the final expression is at least  $(\puni(F) + \eps)\abs{J_0}\abs{J_1}$.
    Thus, $H$ is $(\puni(F)+\eps, \delta)$-interval-dense. Also $\abs{V(H)}\geq 2^k\geq N$, which means that $H$ contains $F$. But $H$ is a subgraph of $G$ so $G$ contains $F$, a contradiction.
\end{proof}

\section{Palette density versus uniformity}\label{sec:pal-uni}
In this section we follow an argument of Lamaison~\cite{lamaison2024palettesdetermineuniformturan} to prove that $\ppal(F)\geq \puni(F)$. 
We need to extend the definition of the density between two vertex-sets to allow for multisets, as follows.
Let $G$ be a graph and let 
$\mathbf{x} = (x_1,\dots,x_s)$ and 
$\mathbf{y} = (y_1,\dots,y_t)$ be two vectors of vertices of $G$ (i.e., $x_i,y_j \in V(G)$ for all $i, j$). We define the \defn{density} between $\mathbf{x}$ and $\mathbf{y}$ to be
\begin{equation*}
d(\mathbf{x},\mathbf{y}) \coloneqq \frac{1}{st} \cdot \abs{\set{(i,j) \in [s] \times [t] : x_iy_j \in E(G)}}.
\end{equation*}

The main lemma of this section is as follows.

\begin{lemma}\label{lem:palette main}
For every $\varepsilon > 0$ there exists an integer $s \geq 1$, such that for every $m \geq 1$, if $N$ is large enough as a function of $\varepsilon,m$, then the following holds: Let $\alpha > 0$, let $V_1,\dots,V_N$ be pairwise-disjoint vertex-sets in a graph, and suppose that $d(V_i,V_j) \geq \alpha$ for all $1 \leq i < j \leq N$. Then there  
is a subset $I \subseteq [N]$ of size $m$ and there are length-$s$ vectors $\mathbf{x}^{(i)} \in V_i^s$ for every 
$i \in I$, such that
$d(\mathbf{x}^{(i)},\mathbf{x}^{(j)}) \geq \alpha - \varepsilon$ for every pair of distinct $i,j \in I$. 
\end{lemma}

Before proving \cref{lem:palette main} we shall show that it implies $\ppal(F) \geq \puni(F)$.

\begin{theorem}\label{thm:pal >= uni}
    For every ordered graph $F$, $\ppal(F) \geq \puni(F)$.
\end{theorem}

\begin{proof}
Let $\varepsilon > 0$ be arbitrary, and fix parameters
$
\frac{1}{\varepsilon} \ll s \ll m \ll N.
$
Let $G$ be a $(\ppal(F) + 2\varepsilon, \frac{1}{2N})$-interval-dense ordered graph on $n \geq N$ vertices, and let us show that $G$ contains a copy of $F$. As $\varepsilon$ is arbitrary, this would show that $\ppal(F) \geq \puni(F)$. Let $V_1 < \dots < V_N$ be disjoint intervals of $V(G)$ with $\abs{V_i} \geq \lfloor n/N \rfloor \geq n/(2N)$ for all $i \in [N]$. By the definition of interval-density, we have
$d(V_i,V_j) \geq \ppal(F) + 2\varepsilon$ for all 
$1 \leq i < j \leq N$. By 
\cref{lem:palette main} (with $\alpha \coloneqq \ppal(F) + 2\varepsilon$), there is an index-set $I \subseteq [N]$ of size $m$ and
length-$s$ vectors $\mathbf{x}^{(i)} \in V_i^s$ for every 
$i \in I$, such that
\begin{equation}\label{eq:density between s-tuples}
d(\mathbf{x}^{(i)},\mathbf{x}^{(j)}) \geq 
\ppal(F) + \varepsilon
\end{equation}
for all distinct $i, j \in I$. For each pair $i, j \in I$ with $i < j$, define
\begin{equation*}
B_{ij} \coloneqq \set[\big]{(k,\ell) \in [s]^2 : \mathbf{x}^{(i)}_k\mathbf{x}^{(j)}_{\ell} \in E(G) }.
\end{equation*}
Then \eqref{eq:density between s-tuples} implies that 
$\abs{B_{ij}}/s^2 \geq 
\ppal(F) + \varepsilon$.

Next, note that the number of possible values that $B_{ij}$ can take is at most $2^{s^2}$. Thus, if $m = \abs{I}$ is large enough, then by the multicolour Ramsey theorem (with $2^{s^2}$ colours), there exists $J \subseteq I$ with $\abs{J} = \abs{V(F)} \eqqcolon f$ and a subset $B \subseteq [s]^2$ such that $B_{ij} = B$ for all $i,j \in J$, $i < j$. Without loss of generality, $J = [f]$. 

Now, let $D$ be the digraph on the vertex-set $[s]$ with edge-set corresponding to $B$, with loops deleted; i.e., for distinct $k,\ell \in [s]$, $k\ell \in A(D)$ if and only if $(k,\ell) \in B$. Then 
\begin{equation*}
\cL(D) \geq \frac{\abs{A(D)}}{\abs{V(D)}^2} \geq 
\frac{\abs{B} - s}{s^2} \geq 
\ppal(F) + \varepsilon - \frac{1}{s} \geq \ppal(F) + \frac{\varepsilon}{2}.
\end{equation*}
By the definition of $\ppal$, $F$ is embeddable into $D$. So let $\varphi \colon V(F) \to V(D) = [s]$ be such an embedding, and suppose for simplicity that $V(F) = [f]$. 
Namely, $\varphi$ satisfies that $(\varphi(i),\varphi(j)) \in A(D)$ for every edge $ij \in E(F)$ with $i < j$. 
Consider the $f$ vertices 
$x_i \coloneqq \mathbf{x}^{(i)}_{\varphi(i)} \in V_i$, $i = 1,\dots,f$. We have 
$x_1 < \dots < x_f$. Also, for every edge $ij \in E(F)$ with $i < j$, we have 
$(\varphi(i),\varphi(j)) \in A(D) \subseteq B$ and hence 
$x_ix_j = \mathbf{x}^{(i)}_{\varphi(i)}\mathbf{x}^{(j)}_{\varphi(j)} \in E(G)$ by the definition of $B_{ij} = B$. But this means that $x_1,\dots,x_f$ form a copy of $F$ in $G$, as required.
\end{proof}

It remains to prove \cref{lem:palette main}. We will need the following auxiliary lemma.
\begin{lemma}\label{lem:palette aux}
For every $m$ there exists $n$ such that the following holds. Let $W_1,\dots,W_n$ be finite non-empty pairwise-disjoint sets. For all $1 \leq i \neq j \leq n$, let a set $B_i^{(j)} \subseteq W_i$ be given with 
$\abs{B_i^{(j)}} \leq 0.1\abs{W_i}$. Then there is an index-set $I \subseteq [n]$ with $\abs{I} = m$ such that for all $i \in I$, 
\begin{equation*}
W_i \setminus \bigcup_{j \in I \setminus \set{i}} B_i^{(j)} \neq \emptyset.
\end{equation*}
\end{lemma}
\begin{proof}
    Suppose otherwise. Then for every 
    $I = \set{i_1 < \dots < i_m} \subseteq [n]$, there exists
    $k = k(I) \in [m]$ such that 
    \begin{equation*}
    W_{i_k} = \bigcup_{j \in [m] \setminus \set{k}} B_{i_k}^{(i_j)} \; \; .
    \end{equation*}
    Consider the $m$-colouring of $\binom{[n]}{m}$ where each set $I$ is coloured with colour $k(I)$. If $n$ is large enough then by the hypergraph Ramsey theorem, there exists 
    $J \subseteq [n]$ with $\abs{J} = 4m$, as well as an index 
    $k \in [m]$, such that $k(I) = k$ for all 
    $I \in \binom{J}{m}$. Without loss of generality, 
    $J = \set{1,\dots,4m}$. 
    We have 
    \begin{equation*}
    \sum_{j \in J \setminus \set{2m}} \abs{B_{2m}^{(j)}} \leq 
    4m \cdot 0.1\abs{W_{2m}} \leq 0.5m \cdot \abs{W_{2m}}.
    \end{equation*}
    Hence, there are at most $0.5\abs{W_{2m}}$ vertices in 
    $W_{2m}$ which belong to at least $m$ of the sets $B_{2m}^{(j)}$, 
    $j \in J \setminus \set{2m}$. Fix a vertex $w \in W_{2m}$ which belongs to at most $m-1$ of these sets. Then there are at least $m-1$ indices $j \in \set{1,\dots,2m-1}$ such that $w \notin B_{2m}^{(j)}$, as well as at least $m-1$ indices 
    $j \in \set{2m+1,\dots,4m}$ such that 
    $w \notin B_{2m}^{(j)}$. This means that we can choose a subset $I = \set{i_1 < \dots < i_m} \subseteq J$ with $i_k = 2m$ and such that $w \notin B_{2m}^{(i_{\ell})}$ for all $\ell \in [m] \setminus \set{k}$. However, this contradicts the fact that $k(I) = k$.  
\end{proof}

We will also need the following version of Hoeffding's inequality.
\begin{lemma}[Hoeffding's inequality]\label{lem:Hoeffding}
Let $V$ be a finite non-empty set, let $f \colon V \to [0,1]$ be a function, and put $\mu \coloneqq \frac{1}{\abs{V}}\sum_{x \in V} f(x)$.
Let $t \geq 1$, and sample $x_1,\dots,x_t \in V$ uniformly at random and independently. Then for all $\varepsilon > 0$, we have that
\begin{equation*}
\bP\biggl[ \frac{1}{t}\sum_{i=1}^t f(x_i) < \mu - \varepsilon \biggr] \leq 
e^{-2\varepsilon^2 t}.
\end{equation*}
\end{lemma}

\begin{proof}[Proof of \cref{lem:palette main}]
Set
\begin{equation*}
r \coloneqq  \ceil{3/\varepsilon}, \qquad
t \coloneqq \ceil{5\log(10r)/\varepsilon^2},
\end{equation*}
and $s \coloneqq rt$. For an integer $m \geq 1$, let $n(m)$ be the integer given by  \cref{lem:palette aux}. We define a sequence of integers $N_0 > N_1 > \dots > N_r$ by setting $N_r \coloneqq m$ and $N_{\ell-1} = n(N_\ell)$ for 
$\ell = r,r-1,\dots,1$. We show that the statement of \cref{lem:palette main} holds for $N \coloneqq N_0$. So let $V_1,\dots,V_N$ be pairwise-disjoint vertex-sets such that $d(V_i,V_j) \geq \alpha$ for each $1 \leq i < j \leq \nolinebreak N$.

\begin{claim}
    For every $0 \leq \ell \leq r$, there exists an index-set $I_{\ell} \subseteq [N]$ with $\abs{I_{\ell}} = N_{\ell}$, and there exist length-$t$ vectors 
    $\mathbf{x}^{(i,1)},\dots,\mathbf{x}^{(i,\ell)} \in V_i^t$ for each $i \in I_\ell$, such that the following hold\textup{:}
    \begin{enumerate}
        \item For all $1 \leq h \leq \ell$ and distinct $i,j \in I_{\ell}$, 
        $d(\mathbf{x}^{(i,h)}, V_j) \geq 
        \alpha - \frac{\varepsilon}{3}$.
        \item For all $1 \leq h \neq k \leq \ell$ and distinct 
        $i,j \in I_\ell$, $d(\mathbf{x}^{(i,h)},\mathbf{x}^{(j,k)}) \geq 
        \alpha - \frac{2\varepsilon}{3}$.
    \end{enumerate}
\end{claim}
\begin{proof}
    The proof is by induction on $\ell$, and the base case $\ell = 0$ holds vacuously. Let now $\ell \geq 1$. By the induction hypothesis, there exists an index-set 
    $I_{\ell-1} \subseteq [N]$ with $\abs{I_{\ell - 1}} = N_{\ell-1}$, and there exist length-$t$ vectors 
    $\mathbf{x}^{(i,1)},\dots,\mathbf{x}^{(i,\ell-1)} \in V_i^t$ for each $i \in I_{\ell-1}$ satisfying properties 1 and 2. 
    
    For each pair of distinct indices $i,j \in I_{\ell-1}$, we define a subset $B_i^{(j)} \subseteq V_i^t$ as follows: a vector $\mathbf{x} \in V_i^t$ belongs to $B_i^{(j)}$ if 
    $
    d(\mathbf{x},V_j) < \alpha - \frac{\varepsilon}{3}
    $
    or
    $
    d(\mathbf{x},\mathbf{x}^{(j,k)}) < \alpha - \frac{2\varepsilon}{3}
    $
    for some $1 \leq k \leq \ell-1$. 

    We claim that 
    $\abs{B_i^{(j)}} \leq 0.1\abs{V_i}^t$. Indeed, choose $\mathbf{x} = (x_1,\dots,x_t) \in V_i^t$ uniformly at random. Consider first the quantity $d(\mathbf{x},V_j)$. We have $\frac{1}{\abs{V_i}}\sum_{x \in V_i} d(x,V_j) = d(V_i,V_j) \geq \alpha$
    and 
    $\frac{1}{t}\sum_{i=1}^t d(x_i,V_j) = d(\mathbf{x},V_j)$. Hence, by \cref{lem:Hoeffding} with parameter $\frac{\varepsilon}{3}$, we have that 
    $d(\mathbf{x},V_j) \geq \alpha - \frac{\varepsilon}{3}$ holds with probability at least $1 - e^{-2(\varepsilon/3)^2 t} \geq 1 - e^{-\log(10r)} = 1 - \frac{1}{10r}$.
    Next, fix any $1 \leq k \leq \ell-1$ and consider the quantity $d(\mathbf{x},\mathbf{x}^{(j,k)})$. Note that 
    \begin{equation*}
    \frac{1}{\abs{V_i}} \sum_{x \in V_i} d(x,\mathbf{x}^{(j,k)}) = d(V_i,\mathbf{x}^{(j,k)}) \geq 
    \alpha - \frac{\varepsilon}{3},
    \end{equation*}
    using the induction hypothesis. Also, 
    $\frac{1}{t}\sum_{i=1}^t d(x_i,\mathbf{x}^{(j,k)}) = d(\mathbf{x},\mathbf{x}^{(j,k)})$. Hence, by \cref{lem:Hoeffding} with parameter $\frac{\varepsilon}{3}$, we have that 
    $d(\mathbf{x},\mathbf{x}^{(j,k)}) \geq \alpha - \frac{2\varepsilon}{3}$ holds with probability at least $1 - e^{-2(\varepsilon/3)^2 t} \geq 1 - e^{-\log(10r)} = 1 - \frac{1}{10r}$. Since there are $\ell \leq r$ bad events in total, the union bound gives $\abs{B_i^{(j)}} \leq 
    r \cdot \frac{1}{10r} \cdot \abs{V_i}^t = 0.1 \abs{V_i}^t$, as claimed.

    We will now apply \cref{lem:palette aux} with $W_i \coloneqq V_i^t$.
    As $\abs{I_{\ell - 1}} = N_{\ell-1} = n(N_{\ell})$, \cref{lem:palette aux} gives a subset 
    $I_{\ell} \subseteq I_{\ell-1}$ with 
    $\abs{I_{\ell}} = N_{\ell}$ such that for every $i \in I_{\ell}$, the set 
    $V_i^t \setminus \bigcup_{j \in I_{\ell} \setminus \set{i}} B_i^{(j)}$ is non-empty. Now take $\mathbf{x}^{(i,\ell)}$ to be any element of this set. It is easy to see that Items 1-2 are satisfied, completing the induction step.  
\end{proof}

Applying the claim with $\ell = r$, we obtain an index-set $I_r \subseteq [N]$ with $\abs{I_r} = N_r = m$, as well as length-$t$ vectors $\mathbf{x}^{(i,1)},\dots,\mathbf{x}^{(i,r)} \in V_i^t$ for each $i \in I_r$, such that 
$d(\mathbf{x}^{(i,h)},\mathbf{x}^{(j,k)}) \geq 
\alpha - \frac{2\varepsilon}{3}$ 
for all distinct $i,j \in I_r$ and all $1 \leq h \neq k \leq r$. Now, let $\mathbf{x}^{(i)}$ be the concatenation of 
$\mathbf{x}^{(i,1)},\dots,\mathbf{x}^{(i,r)}$;
so $\mathbf{x}^{(i)} \in V_i^s$ (as $s = rt$). For all distinct $i,j \in I_r$, we have
\begin{equation*}
d(\mathbf{x}^{(i)},\mathbf{x}^{(j)}) \geq 
\frac{1}{s^2} \cdot r(r-1) \cdot 
\biggl( \alpha - \frac{2\varepsilon}{3} \biggr)t^2 \geq 
\alpha - \frac{2\varepsilon}{3} - \frac{1}{r} \geq 
\alpha - \varepsilon,
\end{equation*}
using the choice of $r$. This completes the proof of the lemma.
\end{proof}

\section{Properties of minimal palettes}\label{sec:minimal palettes}

We call a digraph $D$ \defn{minimal} if $\cL(D) > \cL(D')$ for every proper 
subdigraph $D'$ of $D$ with at least one vertex.
Note that when determining the palette density $\ppal(F)$ of an ordered graph $F$, we only need to consider minimal palettes. 
We will say that $D$ is a \defn{minimal palette} for $F$ if $D$ is a palette for $F$ and is minimal and so
\[
\ppal(F) = \sup\set{\cL(D) \colon D \text{ is a minimal palette for $F$}}.
\]
In this section we establish some properties of minimal digraphs and use these to prove \cref{thm:rho rational,thm:edge-critical,{thm:rho_chromatic_equiv}}.
These properties may also be useful for studying \cref{prob:spectrum}. 

\begin{definition}[$A_D$]
For a digraph $D$, let \defn{$A_D$} be the \emph{symmetric} $\set{0, 1/2, 1}$-matrix: $(A_D)_{ij}$ is the number of arcs between $i$ and $j$ (in either direction) divided by 2.    
\end{definition}
Recall the definition of the Lagrangian polynomial and observe that
\begin{equation}\label{eq:Lagrangian quadratic form}
\cL(D, \boldsymbol{x}) = \boldsymbol{x}^\top A_D \boldsymbol{x}.
\end{equation}

We now show that minimal digraphs have a very rigid linear-algebraic structure. 

\begin{lemma}\label{lem:palette-linear-algebra}
	A digraph $D$ with $\abs{V(D)} \geq 2$ is minimal if and only if $A_D$ satisfies the following three conditions.
	\begin{enumerate}[noitemsep, label=\textup{(\alph{*})}]
		\item\label{part:invertible} $A_D$ is invertible.
		\item\label{part:positive} All entries of $A_D^{-1} \boldsymbol{1}$ are positive.
		\item\label{part:negative-def} $\boldsymbol{v}^\top A_D \boldsymbol{v} < 0$ for every non-zero $\boldsymbol{v} \perp \boldsymbol{1}$.
	\end{enumerate}
    Furthermore, if $D$ is minimal then $\cL(D, \bullet)$ has a unique maximiser $\boldsymbol{x}_{\ast}$, and we have 
    $A_D \boldsymbol{x}_\ast = \cL(D) \cdot \boldsymbol{1}$ and 
    $1/\cL(D) = 
    \boldsymbol{1}^\top A_D^{-1} \boldsymbol{1}$.
\end{lemma}

\begin{proof}
    Call a vector $\boldsymbol{v}$ \defn{feasible} if $\boldsymbol{v} \geq 0$ and $\boldsymbol{v}^\top \boldsymbol{1} = 1$ (in other words, $\boldsymbol{v}$ is a probability vector).
    Let $P$ be the problem ``maximise $\boldsymbol{x}^\top A_D \boldsymbol{x}$ such that $\boldsymbol{x}$ is feasible'' (this is precisely the Lagrangian). 
    Let $D$ be a minimal digraph. Then no solution of $P$ has a zero entry (otherwise deleting the corresponding vertex from $D$ gives a subdigraph with the same Lagrangian). Since $A_D$ is symmetric, for every real $\lambda$ and vector $\boldsymbol{t}$, we have by \eqref{eq:Lagrangian quadratic form} that
    \begin{equation}\label{eq:expansion}
        \cL(D, \boldsymbol{x} + \lambda \boldsymbol{t}) = \cL(D, \boldsymbol{x}) + 2\lambda \cdot \boldsymbol{x}^\top A_D \boldsymbol{t} + \lambda^2 \cdot \boldsymbol{t}^\top A_D \boldsymbol{t},
    \end{equation}
    which is a polynomial in $\lambda$ of degree at most 2. Suppose that distinct $\boldsymbol{x}$ and $\boldsymbol{y}$ solve $P$ (i.e., $\boldsymbol{x},\boldsymbol{y}$ are maximisers of $\cL(D, \bullet)$). Consider the line through $\boldsymbol{x}$ and $\boldsymbol{y}$ (by this we mean the whole line, and not just the segment between $\boldsymbol{x}$ and $\boldsymbol{y}$). We know that 
    $\cL(D, \bullet)$ is of degree at most two on this line. Since it is maximised at two distinct places, it is constant on the line. But this line contains some feasible $\boldsymbol{v}$ that has a zero entry (namely, any intersection point of the line with the boundary of the simplex of feasible vectors). This contradicts the minimality of $D$. So there is a unique solution $\boldsymbol{x}_\ast$ of $P$ and all its coordinates are in $(0, 1)$.

    Consider \eqref{eq:expansion} with $\boldsymbol{x} = \boldsymbol{x}_\ast$. 
    Suppose there is some non-zero $\boldsymbol{t}$ with $A_D \boldsymbol{t} = \boldsymbol{0}$. Then $\cL(D, \boldsymbol{x}_\ast + \lambda \boldsymbol{t}) = \cL(D, \boldsymbol{x}_\ast)$. 
    For some choice of $\lambda$ (either small positive or small negative), $\boldsymbol{x}_\ast + \lambda \boldsymbol{t}$ will have non-negative entries that sum to at most one. Scaling up gives a new feasible $\boldsymbol{x}'$ with $\cL(D, \boldsymbol{x}') \geq \cL(D, \boldsymbol{x}_\ast)$, a contradiction. Thus $A_D$ is invertible, proving \ref{part:invertible}. 
    Now fix any non-zero $\boldsymbol{t} \perp \boldsymbol{1}$. 
    For all $\lambda$ sufficiently close to 0, $\boldsymbol{x}_\ast + \lambda \boldsymbol{t}$ is feasible. 
    Indeed, all coordinates are in $[0,1]$ if $\abs{\lambda}$ is small enough, and the sum of coordinates is 1 because 
    $\boldsymbol{t} \perp \boldsymbol{1}$.
    Since $\boldsymbol{x}_\ast$ is the unique maximiser, $2 \lambda \cdot \boldsymbol{x}_\ast^\top A_D \boldsymbol{t} + \lambda^2 \cdot \boldsymbol{t}^\top A_D \boldsymbol{t} < 0$ for all sufficiently small non-zero $\lambda$.
    This means that 
    $\boldsymbol{x}_\ast^\top A_D \boldsymbol{t} = 0$ and $\boldsymbol{t}^\top A_D \boldsymbol{t} < 0$. In particular, this proves \ref{part:negative-def}. Since $A_D$ is symmetric, $0 = (\boldsymbol{x}_\ast^\top A_D \boldsymbol{t})^\top = \boldsymbol{t}^\top A_D \boldsymbol{x}_\ast$. Thus, $A_D \boldsymbol{x}_\ast$ is perpendicular to every vector that is perpendicular to $\boldsymbol{1}$ and so $A_D \boldsymbol{x}_\ast \parallel \boldsymbol{1}$. Moreover, $A_D \boldsymbol{x}_\ast$ must be a positive multiple of $\boldsymbol{1}$ because all entries of $A_D$ and $\boldsymbol{x}_\ast$ are non-negative. Thus $A_D^{-1} \boldsymbol{1}$ is a positive multiple of $\boldsymbol{x}_\ast$, which must have positive entries. This proves \nolinebreak \ref{part:positive}. 
    
    Now write 
    $A_D\boldsymbol{x}_\ast = c \cdot \boldsymbol{1}$ for some constant $c$. We have 
    $\cL(D) = \cL(D,\boldsymbol{x}_\ast) =
    \boldsymbol{x}_\ast^\top A_D \boldsymbol{x}_\ast = c \boldsymbol{x}_\ast^\top \boldsymbol{1} = c$, so $A_D \boldsymbol{x}_\ast = \cL(D) \cdot \boldsymbol{1}$. Also,
    $\boldsymbol{1}^\top A_D^{-1} \boldsymbol{1} = \frac{1}{c} \cdot \boldsymbol{1}^\top \boldsymbol{x}_\ast = \frac{1}{c} = 
    \frac{1}{\cL(D)}$.
    So far we proved that if $D$ is minimal, then \ref{part:invertible}-\ref{part:negative-def} hold, as well as the ``Furthermore" part of the lemma.

    Finally, suppose that $A_D$ satisfies \ref{part:invertible}-\ref{part:negative-def}.
    By \ref{part:positive}, the vector
    $\boldsymbol{x} \coloneqq A_D^{-1} \boldsymbol{1}$ has positive entries. By normalising $\boldsymbol{x}$, we obtain a feasible vector
    $\boldsymbol{x}_\ast$ such that
    $A_D \boldsymbol{x}_\ast \parallel 
    A_D \boldsymbol{x} = \boldsymbol{1}$. Let $\boldsymbol{y}$ be any other feasible vector. Then 
    $\boldsymbol{t} \coloneq \boldsymbol{y} - \boldsymbol{x}_\ast \perp \boldsymbol{1}$ and so $\boldsymbol{t}^\top A_D \boldsymbol{x}_\ast = 0$ and $\boldsymbol{t}^\top A_D \boldsymbol{t} < 0$, where the latter holds by \ref{part:negative-def}. But then \eqref{eq:expansion} guarantees that $\cL(D, \boldsymbol{y}) < \cL(D, \boldsymbol{x}_\ast)$, and hence $\boldsymbol{x}_\ast$ is the unique maximiser. As all entries of $\boldsymbol{x}_\ast$ are non-zero, every proper subdigraph $D'$ of $D$ has $\cL(D') < \cL(D)$, and so $D$ is minimal.
\end{proof}

\Cref{lem:palette-linear-algebra} implies the following:
\begin{corollary}\label{cor:minimality test}
    There is an algorithm for testing if a digraph $D$ is minimal.
\end{corollary}
\begin{proof}
    If $\abs{V(D)} = 1$, then $D$ is the single vertex and so minimal.
    Otherwise $n = \abs{V(D)} \geq 2$ and the algorithm needs to test that $D$ satisfies conditions~\ref{part:invertible}-\ref{part:negative-def} of \cref{lem:palette-linear-algebra}. For \ref{part:invertible} and \ref{part:positive} this is easy. 
    Condition~\ref{part:negative-def} can be tested as follows. Let $Z$ be any $n \times (n - 1)$ matrix corresponding to a surjective linear map $\bR^{n - 1} \to \set{v \in \bR^n \colon v \perp \boldsymbol{1}}$. Then condition \ref{part:negative-def} is equivalent to $Z^\top A_D Z$ being negative definite.
\end{proof}

By \cref{lem:palette-linear-algebra}, being a minimal digraph is a property of $A_D$ rather than of $D$. That is, it is a property of the number of arcs between each pair of vertices rather than the direction of those arcs. Hence, when discussing minimal digraphs we may view them as weighted graphs (where each edge has weight 1 or 2).

Condition \ref{part:negative-def} of 
\cref{lem:palette-linear-algebra} is particularly useful as it is hereditary: if $D$ fails \ref{part:negative-def} and $D$ is an induced subdigraph of $D'$, then $D'$ also fails \ref{part:negative-def}. Indeed, take a vector $\boldsymbol{v}$ violating \ref{part:negative-def}, i.e., a non-zero vector $\boldsymbol{v}$ with 
$\boldsymbol{v} \perp \boldsymbol{1}$ and $\boldsymbol{v}^\top A_D \boldsymbol{v} \geq 0$, and pad it with zeroes to obtain a vector $\boldsymbol{v'}$ for which $\boldsymbol{v'}^\top A_{D'} \boldsymbol{v'} \geq 0$.
Using \cref{lem:palette-linear-algebra}, we now derive the following basic properties of minimal digraphs. By an 
\defn{optimal weighting} of $D$ we mean a feasible vector maximising 
$\cL(D, \bullet)$. For a feasible vector $x$, the \defn{weighted degree} of vertex 
$i \in V(D)$ is the sum of $x_j$ over all edges $(i,j),(j,i) \in A(D)$.

\begin{lemma}\label{cor:palette-basics}
    Let $D$ be a minimal digraph. 
    \begin{enumerate}[noitemsep, label=\textup{(\alph{*})}]
        \item\label{part:complete} There is at least one arc between every pair of vertices of $D$.
        \item\label{part:regular} In the optimal weighting for $D$, every vertex has weighted degree $2 \cL(D)$.
    \end{enumerate}
\end{lemma}

\begin{proof}
    The one-vertex digraph satisfies both conditions so we will assume that $\abs{V(D)} \geq 2$.
    The arcless digraph on two vertices, $D_0$, has $A_{D_0} = \bigl(\begin{smallmatrix} 0 & 0 \\ 0 & 0 \end{smallmatrix}\bigr)$ which fails
    condition \ref{part:negative-def} of \cref{lem:palette-linear-algebra} (with $\boldsymbol{v} = (1, -1)$). Hence every minimal digraph does not contain $D_0$ as an induced subdigraph, which is exactly \ref{part:complete}.

    By the ``Furthermore" part of \cref{lem:palette-linear-algebra}, the unique optimal weighting $\boldsymbol{x}_\ast$ satisfies $A_D \boldsymbol{x}_\ast = \cL(D) \cdot \boldsymbol{1}$. Since the $i$th entry of $A_D \boldsymbol{x}_\ast$ is $\cL(D) = \sum_j (A_D)_{ij} (\boldsymbol{x}_\ast)_j$, which is half the weighted degree of vertex $i$, we have \ref{part:regular}.
\end{proof}

As we saw above, digraph minimality is a property of the matrix $A_D$ rather than of $D$, meaning that what matters is the number of edges between any pair of vertices, and not their direction. 
By \cref{cor:palette-basics}\ref{part:complete}, 
this number of edges is always 1 or 2. 
Hence, in order to specify $A_D$, we need only specify the (undirected) graph, \defn{$D^{(2)}$}, consisting of the bidirectional edges. We now use the hereditary property \cref{lem:palette-linear-algebra}\ref{part:negative-def} to deduce structural properties of $D^{(2)}$. 

\begin{lemma}\label{lem:palette-connected}
    Let $D$ be a minimal digraph.
    \begin{enumerate}[noitemsep, label=\textup{(\alph{*})}]
        \item\label{part:2K2} $D^{(2)}$ does not contain an induced $2K_2$.
        \item\label{part:span} Either $D^{(2)}$ has no edges or it is a connected graph that spans $V(D)$.
    \end{enumerate}
\end{lemma}

\begin{proof}
    For \ref{part:2K2}, 
    suppose that four vertices of $D^{(2)}$ induce a $2K_2$ and let $D_0$ be the corresponding induced subdigraph of $D$. By \cref{cor:palette-basics}\ref{part:complete} its matrix is
    \[
        A_{D_0} = \Biggl(
        \begin{smallmatrix}
            0 & 1 & 1/2 & 1/2 \\
            1 & 0 & 1/2 & 1/2 \\
            1/2 & 1/2 & 0 & 1 \\
            1/2 & 1/2 & 1 & 0
        \end{smallmatrix}
        \Biggr).
    \]
    This matrix fails \cref{lem:palette-linear-algebra}\ref{part:negative-def} (with $\boldsymbol{v} = (1, 1, -1, -1)$) and so we have \ref{part:2K2}.

    For \ref{part:span}, let $D$ be a minimal digraph, and consider the connected components of $D^{(2)}$. First suppose that $D^{(2)}$ has an isolated vertex $v$. Then, the weighted degree of $v$ in the optimal weighting is 1 minus the weight of $v$ and so less than 1. Thus, by \cref{cor:palette-basics}\ref{part:regular}, we have $\cL(D) < 1/2$. This implies that $D$ has no edges of weight two (since these have Lagrangian $1/2$) and so $D^{(2)}$ has no edges. Otherwise, $D^{(2)}$ has no isolated vertices and so every connected component of $D^{(2)}$ contains an edge. But $D^{(2)}$ has no induced $2K_2$ so has at most one connected component. That is, $D^{(2)}$ is a connected graph spanning $V(D)$, as required.
\end{proof}

Next, we use \cref{cor:palette-basics} to show that the supremum in the definition of the palette density $\ppal$ (see equation~\eqref{eq:palette density}) is attained.

\begin{lemma}\label{lem:palette density supremum attained}
    For every ordered graph $F$, there is a minimal palette $D$ for $F$ such that 
    $\ppal(F) = \cL(D)$. That is, the supremum in \eqref{eq:palette density} is attained.
\end{lemma}
\begin{proof} 
    It suffices to show that the set of minimal palettes for $F$ is finite. Recall that $D$ is a palette for $F$ if $\vec{F} \nrightarrow D$, where $\vec{F}$ is the digraph obtained from $F$ by orienting all edges forward. 
    We claim that a minimal palette for $F$ has fewer than $2^{\abs{V(F)}}$ vertices. Indeed, let $D$ be a minimal digraph with at least $2^{\abs{V(F)}}$ vertices. By \cref{cor:palette-basics}\ref{part:complete}, $D$ contains a tournament on $2^{\abs{V(F)}}$ vertices. Then, by the Ramsey theorem for tournaments, $D$ contains a transitive tournament on $\abs{V(F)}$ vertices. Since $\vec{F}$ is acyclic, we get that $\vec{F} \rightarrow D$ and so $D$ is not a palette for $F$, as required.
\end{proof}

Using \cref{lem:palette density supremum attained}, we can prove \cref{thm:rho rational}.
\begin{proof}[Proof of \cref{thm:rho rational}]
To prove rationality, it suffices to show that $\ppal(F)$ is rational for every ordered graph $F$ (by \cref{thm:palette}). By \cref{lem:palette density supremum attained}, it then suffices to show that $\cL(D)$ is rational for every minimal digraph $D$.
If $D$ is the 1-vertex graph then 
$\cL(D) = 0$ and we are done.
Else, by the ``Furthermore" part of \cref{lem:palette-linear-algebra}, 
we have $1/\cL(D) = \boldsymbol{1}^\top A_D^{-1} \boldsymbol{1}$. The RHS is rational because $A_D$ is an invertible rational matrix.

For the second part of the theorem, the relative Tur\'an density of a given ordered graph $F$ can be computed as follows.  
First, enumerate all digraphs on at least two and fewer than $2^{\abs{V(F)}}$ vertices, test if each such digraph $D$ is minimal (using \cref{cor:minimality test}) and satisfies $\vec{F} \nrightarrow D$. By the proof of \cref{lem:palette density supremum attained}, the digraphs which pass these tests are exactly the minimal palettes for $F$ with at least two vertices.
Finally, output the maximum value of $\cL(D)$ over all these minimal palettes (if there are none, then output 0) using the formula $1/\cL(D) = 
\boldsymbol{1}^\top A_D^{-1} \boldsymbol{1}$ from \cref{lem:palette-linear-algebra}.
\end{proof}

Next, we prove \cref{thm:edge-critical}.
\begin{proof}[Proof of \cref{thm:edge-critical}]
    Let $F$ be an ordered graph such that $\chi(F) = 3$ and there is an edge $e \in E(F)$ with $\chi(F-e) = 2$. We need to show that $\ppal(F) = \frac{1}{2}$. First, $\ppal(F) \geq \frac{1}{2}$ because $F$ is not embeddable into the 2-vertex digraph consisting of a bidirectional edge (since $F$ is not bipartite). To show that $\ppal(F) = \frac{1}{2}$, let $D$ be a minimal palette for $F$ with $\ppal(F) = \cL(D)$ (such a palette exists by \cref{lem:palette density supremum attained}).
    By \cref{lem:palette-connected}\ref{part:span}, either $D$ has no bidirectional edges, $D$ is the bidirectional edge, or $D$ contains a triangle $xyz$ where $xy, xz$ are bidirectional. We show that the first and last cases are impossible. 
    The first case is impossible because if $D$ has no bidirectional edges then 
    $\cL(D) < \frac{1}{2} \leq \ppal(F)$. For the third case, suppose by contradiction that $xy,xz$ are bidirectional.
    By \cref{cor:palette-basics}\ref{part:complete}, at least one edge is present between $y$ and $z$. 
    Let $A \cup B$ be a bipartition of 
    $F - e$, say with $e \subseteq B$. 
    Map $A$ to $\set{x}$ and $B$ to 
    $\set{y, z}$ such that the endpoints of $e$ are mapped to $y,z$. It is easy to see that this gives a homomorphism from $\vec{F}$ to $D$, contradicting that $D$ is a palette for $F$. It follows that $D$ is a bidirectional edge,
    and hence 
    $\ppal(F) = \cL(D) = \frac{1}{2}$.
\end{proof}

Finally, we prove \cref{thm:rho_chromatic_equiv}.

\begin{proof}[Proof of \cref{thm:rho_chromatic_equiv}]
    From~\eqref{eq:pi-upper}, we know that $\rord (F) \le 1 - \frac{1}{\cord(F) - 1}$. 
    Now we show that 
    $\rord(F) \ge 
    1 - \frac{1}{\chi(F) - 1}$. 
    By \cref{thm:palette}, it suffices to show that 
    $\ppal(F) \geq 
    1 - \frac{1}{\chi(F) - 1}$. To see this, let $D$ be the digraph with $\chi(F)-1$ vertices and a bidirectional edge between every two vertices. Then $\cL(D) = 1 - \frac{1}{\chi(F) - 1}$ and $D$ is a palette for $F$ (since $\chi(D) < \chi(F)$). We conclude that if $\cord(F) = \chi(F)$, then 
    $\rord(F) = 1 - \frac{1}{\cord(F) - 1}$.

    Now suppose that $\rord(F) = 1 - \frac{1}{\cord(F) - 1}$. By \cref{thm:palette} and \cref{lem:palette density supremum attained}, we may fix a minimal palette $D$ for $F$ with $\cL(D) = 1 - \frac{1}{\cord(F) - 1}$. Let $\boldsymbol{x} = (x_v \colon v \in V(D))$ be an optimal weighting of $D$. By \cref{lem:palette-linear-algebra}, every $x_v$ is positive and, by \cref{cor:palette-basics}\ref{part:complete}, there is at least one arc between every pair of vertices of $D$.

    For a linear order $\prec$ of $V(D)$, define an undirected graph \defn{$H_{\prec}$} on $V(D)$ where $uv \in E(H_{\prec})$ if the arc from the $\prec$-earlier vertex of $\set{u, v}$ to the $\prec$-later vertex is in $D$. 
    In other words, $H_{\prec}$ is the graph of forward edges of $D$ under $\prec$. 
    The definition of interval chromatic number implies that $F$ is embeddable into the transitive tournament on $\cord(F)$ vertices and so, since $D$ is a palette for $F$, the graph $H_{\prec}$ is $K_{\cord(F)}$-free.
    Define 
    \[
    \cL(H_\prec, \boldsymbol{x}) \coloneqq \sum_{uv \in E(H_\prec)} x_u x_v.
    \]
    The weighted Tur\'{a}n theorem (see \cite{motzkin1965maxima}) implies that 
    \[
    \cL(H_\prec, \boldsymbol{x}) \leq \frac{1}{2} \cdot 
    \biggl( 1 - \frac{1}{\cord(F) - 1} \biggr).
    \]
    On the other hand, the average of $\cL(H_\prec, \boldsymbol{x})$ over all possible linear orders is
    \[
    \frac{1}{2} \sum_{(u, v) \in A(D)} x_u x_v = \frac{1}{2} \cL(D) = \frac{1}{2}\biggl(1 - \frac{1}{\cord(F) - 1} \biggr).
    \]
    Therefore, $\cL(H_\prec, \boldsymbol{x}) = 1/2 \cdot (1 - 1/(\cord(F) - 1))$ for every linear order $\prec$. 

    Now, we claim that $D$ is a bidirectional complete graph. 
    Indeed, suppose towards a contradiction that there is a pair of vertices with exactly one arc between them: $(u, v) \in A(D)$ and $(v, u) \notin A(D)$. Choose a linear order where $u, v$ are consecutive and then interchange them. This changes $\cL(H_\prec, \boldsymbol{x})$ by exactly $x_u x_v > 0$ which is a contradiction.
    Since every pair of vertices has at least one arc between them, this proves that $D$ is indeed a bidirectional complete graph.
    Thus $\cL(D) = 1 - 1/\abs{V(D)}$ and so $\abs{V(D)} = \cord(F) - 1$. 
    
    Note that $F$ is embeddable into a bidirectional $K_{\chi(F)}$ and so, since $D$ is a palette for $F$, $\chi(F) > \abs{V(D)}$ and hence $\chi(F) \geq \cord(F)$. But the chromatic number is always at most the interval chromatic number and so $\cord(F) = \chi(F)$, as required.
\end{proof}

\paragraph{Acknowledgements.} Part of this project was initiated during the inaugural staycation workshop at the London School of Economics. We are grateful to LSE for accommodating us.

{
\fontsize{11pt}{12pt}
\selectfont
	
\hypersetup{linkcolor={red!70!black}}
\setlength{\parskip}{2pt plus 0.3ex minus 0.3ex}

\newcommand{\etalchar}[1]{$^{#1}$}

}

\end{document}